\documentclass[11pt]{amsart}
\usepackage[margin=1.2in]{geometry}
\usepackage{amsmath,amsthm,amssymb,mathtools,amsfonts,amsopn,amscd}
\usepackage{bm}
\usepackage{tikz}
\usepackage{tikz-cd}    
\usepackage{rotating}
\usepackage{commath}    
\usepackage{graphicx}
\usepackage{etoolbox}
\usepackage{enumitem}
\usepackage{longtable}
\usepackage{booktabs} 
\usepackage{tabularx} 
\usepackage{comment}
\usepackage{hyperref}
\usepackage{subfiles}
\usepackage{mathrsfs}
\usepackage{calc}
\usepackage[normalem]{ulem}
\usepackage[numbered]{bookmark}   
\usepackage[sc]{mathpazo}
\usepackage{euscript}

\usepackage[noabbrev]{cleveref}
\crefname{section}{Section}{Sections}
\Crefname{section}{Section}{Sections}
\crefname{subsection}{Section}{Sections}
\Crefname{subsection}{Section}{Sections}
\crefname{equation}{Equation}{Equations}
\Crefname{equation}{Equation}{Equations}
\crefname{figure}{Figure}{Figures}
\Crefname{figure}{Figure}{Figures}
\crefname{table}{Table}{Tables}
\Crefname{table}{Table}{Tables}
\crefname{thm}{Theorem}{Theorems}
\Crefname{thm}{Theorem}{Theorems}
\crefname{lem}{Lemma}{Lemmas}
\Crefname{lem}{Lemma}{Lemmas}
\crefname{prop}{Proposition}{Propositions}
\Crefname{prop}{Proposition}{Propositions}
\crefname{cor}{Corollary}{Corollaries}
\Crefname{cor}{Corollary}{Corollaries}
\crefname{df}{Definition}{Definitions}
\Crefname{df}{Definition}{Definitions}
\crefname{ex}{Example}{Examples}
\Crefname{ex}{Example}{Examples}
\crefname{rmk}{Remark}{Remarks}
\Crefname{rmk}{Remark}{Remarks}
\crefname{clm}{Claim}{Claims}
\Crefname{clm}{Claim}{Claims}
\crefname{conj}{Conjecture}{Conjectures}
\Crefname{conj}{Conjecture}{Conjectures}
\setlist[enumerate,1]{label=(\roman*)}

\AddToHook{env/lem/begin}{\crefalias{thm}{lem}}
\AddToHook{env/prop/begin}{\crefalias{thm}{prop}}
\AddToHook{env/cor/begin}{\crefalias{thm}{cor}}
\AddToHook{env/df/begin}{\crefalias{thm}{df}}
\AddToHook{env/ex/begin}{\crefalias{thm}{ex}}
\AddToHook{env/rmk/begin}{\crefalias{thm}{rmk}}
\AddToHook{env/clm/begin}{\crefalias{thm}{clm}}
\AddToHook{env/conj/begin}{\crefalias{thm}{conj}}

\definecolor{secondaryColor}{RGB}{0, 0, 170}
\hypersetup{
	colorlinks=true,
	linkcolor=blue,
	urlcolor=secondaryColor,
	citecolor=secondaryColor,
	linktoc=page,
}

\theoremstyle{plain}
\newtheorem{thm}{Theorem}[section]
\newtheorem{lem}[thm]{Lemma}
\newtheorem{prop}[thm]{Proposition}
\newtheorem{cor}[thm]{Corollary}
\newtheorem{conj}[thm]{Conjecture}

\theoremstyle{definition}
\newtheorem{df}[thm]{Definition}

\theoremstyle{remark}
\newtheorem{rmk}[thm]{Remark}

\numberwithin{equation}{subsection}

\newcommand{\ZZ}{\mathbb{Z}}

\newcommand{\QQ}{\mathbb{Q}}

\newcommand{\CC}{\mathbb{C}}

\newcommand{\FF}{\mathbb{F}}

\newcommand{\Acal}{\mathcal{A}}
\newcommand{\Bcal}{\mathcal{B}}
\newcommand{\Ccal}{\mathcal{C}}

\newcommand{\Gcal}{\mathcal{G}}

\newcommand{\Lcal}{\mathcal{L}}

\newcommand{\Ocal}{\mathcal{O}}

\newcommand{\Zcal}{\mathcal{Z}}

\newcommand{\GL}{\operatorname{GL}}
\newcommand{\SL}{\operatorname{SL}}

\newcommand{\tr}{\operatorname{tr}}

\newcommand{\ord}{\operatorname{ord}}

\let\oldforall\forall
\renewcommand{\forall}{\oldforall \: }
\let\oldexist\exists
\renewcommand{\exists}{\oldexist \: }

\let\emptyset\varnothing

\usepackage{stmaryrd}
\SetSymbolFont{stmry}{bold}{U}{stmry}{m}{n}

\newcommand{\dep}{\operatorname{dep}}
\newcommand{\wt}{\operatorname{wt}}
\newcommand{\gr}{\operatorname{gr}}

\newcommand{\ww}[1]{\mathfrak{#1}}

\newcommand{\bff}{\bm{f}}

\newcommand{\bfz}{\mathbf{z}}
\newcommand{\bfw}{\mathbf{w}}

\newcommand{\Lamz}{\Lambda_{\bfz'}}
\DeclareMathOperator{\Span}{span}

\newcommand{\ov}[1]{\overline{#1}}
\newcommand{\wtd}[1]{\widetilde{#1}}

\allowdisplaybreaks

\newcommand{\ang}[1]{\langle #1 \rangle}
\renewcommand{\subset}{\subseteq}
\renewcommand{\emptyset}{\varnothing}

\makeatletter

\newcommand{\myToC}{{
		\renewcommand{\contentsname}{}
		\@starttoc{toc}{\contentsname}
}}

\patchcmd{\@tocline}
{\hfil}
{\leaders\hbox{\,.\,}\hfil}

\makeatother

\title[On Multiple Eisenstein Series in Positive Characteristic: Direct Sum Result]{On Multiple Eisenstein Series in Positive Characteristic:\\ Direct Sum Result}

\author{Chieh-Yu Chang}
\address{(Chieh-Yu Chang) Department of Mathematics, National Tsing Hua University, No. 101, Sec. 2, Guangfu Rd., East Dist., Hsinchu City 300044, Taiwan (R.O.C.)}
\email{cychang@math.nthu.edu.tw}

\author{Song-Yun Chen}
\address{(Song-Yun Chen) Department of Mathematics, National Tsing Hua University, No. 101, Sec. 2, Guangfu Rd., East Dist., Hsinchu City 300044, Taiwan (R.O.C.)}
\email{s114021702@m114.nthu.edu.tw}

\author{Fei-Jun Huang}
\address{(Fei-Jun Huang) Department of Mathematics, National Tsing Hua University, No. 101, Sec. 2, Guangfu Rd., East Dist., Hsinchu City 300044, Taiwan (R.O.C.)}
\email{fjhuang@gapp.nthu.edu.tw}

\author{Hung-Chun Tsui}
\address{(Hung-Chun Tsui) Department of Mathematics, National Tsing Hua University, No. 101, Sec. 2, Guangfu Rd., East Dist., Hsinchu City 300044, Taiwan (R.O.C.)}
\email{hctsui@gapp.nthu.edu.tw}

\date{\today}

\thanks{}

\begin{document}
	
\begin{abstract}
In this paper, we study multiple Eisenstein series (MES) in positive characteristic. By computing and analyzing the $t$-expansions of MES, we determine the precise ``weights" of their coefficients. This framework enables us to establish a graded algebra structure for MES, extending the direct sum result for Thakur's multiple zeta values proved in~\cite{Cha14}. Our result may also be viewed as a function field analogue of the corresponding direct sum result of Bachmann and Kanno~\cite{BK26}.
\end{abstract}

\maketitle
\tableofcontents
\section{Introduction}

\subsection{Classical story}
Throughout this paper, we denote by $\ZZ$ the set of integers.
An index refers to the empty index $\emptyset$ or a tuple of positive integers.
For any nonempty index $\ww{s}=(s_1,\ldots,s_m)$, we put 
\[
{\wt}(\ww{s})=s_1+\cdots+s_m
\qquad\text{and}\qquad
{\dep}(\ww{s})=m,
\]
which are called the weight and depth of the index $\ww{s}$, respectively. By convention, we define $\wt(\varnothing)=\dep(\varnothing)=0$.

Classical real multiple zeta values (MZVs for short) are generalizations
of the Riemann zeta values at positive integers, originating
in the works of Euler, Hoffman, and Zagier~\cite{Hoffman92,Zagier94}. They are defined, for any nonempty index $\ww{s}=(s_{1},\ldots,s_{m})$ with $s_{1}\geq 2$, by the following series
\[ \zeta(s_1,\ldots,s_m)
:=
\sum_{n_1>\cdots>n_m\geq 1}
\frac{1}{n_1^{s_1}\cdots n_m^{s_m}} .\] 
MZVs have been intensively studied since their introduction, as they occur in number theory and arithmetic geometry as well as their intersection (see~\cite{And04,Zhao2016,BF26}).

A remarkable deformation of MZVs is provided by the multiple Eisenstein series (abbreviated as MES) initiated by Gangl-Kaneko-Zagier  \cite{gkz2006double} in the depth two case, and generalized by Bachmann for higher depth \cite{bachmann2012multiple}. We denote by $\mathbb{H}$ the complex upper half-plane. For any nonempty
index $\ww{s}=(s_1,\ldots,s_m)$ with $s_1\geq 2,\ldots,s_m\geq 2$, the associated
MES is the multiple lattice sum
\[
G_{\ww{s}}(\tau)
:=
\sum_{\lambda_1\succ\cdots\succ\lambda_m\succ 0}
\frac{1}{\lambda_1^{s_1}\cdots\lambda_m^{s_m}},
\qquad \tau\in\mathbb{H},
\]
where the summation is taken over the lattice $\mathbb{Z}\tau+\mathbb{Z}$.
When $s_1=2$, we use Eisenstein summation (see~\cite{Bachmann2023StuffleRegularized}). Here,
$\lambda=a\tau+b\succ0$ means $a>0$, or $a=0$ and $b>0$, and
$\lambda\succ\lambda'$ means $\lambda-\lambda'\succ0$. Based on \cite{gkz2006double,bachmann2012multiple}, the constant term of the Fourier expansion of $G_{\ww{s}}(\tau)$ is the corresponding MZV $\zeta(\ww{s})$. 

For any integer $w\geq 2$, we denote by $\mathfrak{Z}_{w}$
(resp.~$\mathcal{E}_{w}$) the $\mathbb{Q}$-vector space spanned by
$\zeta(\ww{s})$ for $\ww{s}=(s_1,\ldots,s_m)$ with $s_1\geq 2$
(resp.~$G_{\ww{s}}(\tau)$ for $\ww{s}=(s_1,\ldots,s_m)$ with
$s_1\geq 2,\ldots,s_m\geq 2$) and $\wt(\ww{s})=w$. We further set
\[
\mathfrak{Z}_0=\mathcal{E}_0=\mathbb{Q},
\qquad
\mathfrak{Z}_1=\mathcal{E}_1=\{0\}.
\]
Put $\mathfrak{Z}=\sum_{w\geq 0}\mathfrak{Z}_{w}$ and $\mathcal{E}:=\sum_{w\geq 0}\mathcal{E}_{w}$.  From the harmonic product,  $\mathfrak{Z}$ and $\mathcal{E}$ form $\mathbb{Q}$-algebras. More precisely,  we have that   for integers $w_{1},w_{2}\geq 0$,
\[ \mathfrak{Z}_{w_{1}}\cdot\mathfrak{Z}_{w_2}\subset \mathfrak{Z}_{w_1+w_2}, \quad  \mathcal{E}_{w_{1}}\cdot\mathcal{E}_{w_2}\subset \mathcal{E}_{w_1+w_2}. \]
The celebrated Goncharov--Zagier direct sum conjecture
\cite{Zagier94,Goncharov01} predicts that there should be no nontrivial
$\mathbb{Q}$-linear relations among MZVs of different weights:
\begin{conj} [{\textnormal{Goncharov--Zagier}}]
    The $\mathbb{Q}$-algebra $\mathfrak{Z}$ is a graded $\mathbb{Q}$-algebra, i.e., $\mathfrak{Z}=\bigoplus_{w\geq 0}\mathfrak{Z}_{w}$. 
\end{conj}

In the very recent paper~\cite{BK26}, Bachmann and Kanno give an affirmative answer to the analogue of Goncharov--Zagier's conjecture for MES in the following results.

\begin{thm}[{\textnormal{Bachmann-Kanno~\cite{BK26}}}]\label{Thm:BK26}
    Let $\mathcal{E}_{\CC}$ be the $\CC$-vector space spanned by all multiple Eisenstein series. Then the following results hold.
    \begin{enumerate}
        \item $\mathcal{E}=\bigoplus_{w\geq 0}\mathcal{E}_{w}$.
        \item The natural map $\mathcal{E}\otimes_{\QQ}\CC \twoheadrightarrow \mathcal{E}_{\CC}$ is an isomorphism.
    \end{enumerate}
\end{thm}

In the function field setting, the first named author~\cite{Cha14} proved an analogue of Goncharov--Zagier's direct sum conjecture for Thakur's MZVs by incorporating the Baker--Wüstholz--Yu philosophy \cite{Wus89,BW07,Yu97}. The purpose of this paper is to extend the result of~\cite{Cha14} to MES in positive characteristic defined in~\cite{Chen2017, CCHT25}. 

\subsection{Function field analogue}

Throughout this paper, we fix the following notation. 
Let $p$ be a prime and let $q$ be a power of $p$. Let $\mathbb{F}_q$
denote the finite field with $q$ elements and  $\theta$ be a variable. We denote by $A:=\mathbb{F}_q[\theta]$ the polynomial ring in  $\theta$ over $\FF_{q}$ with field of fractions $K:=\mathbb{F}_q(\theta)$. Let
$K_\infty:=\mathbb{F}_q(\!(1/\theta)\!)$
be the completion of $K$ at the infinite place $\infty$, with uniformizer
$1/\theta$, and let
$\mathbb{C}_\infty:=\widehat{\overline{K_\infty}}$
be the completion of a fixed algebraic closure of $K_\infty$. We let $|\cdot|$ be the absolute value on $\CC_\infty$ normalized such that $|\theta|=q$. We further
denote by $\overline{K}$ the algebraic closure of $K$ in $\mathbb{C}_\infty$. Finally, we put $A_+=\{a\in A\mid a\text{ is monic} \}$.

For any nonempty index $\ww{s} = (s_1,\ldots,s_m)$, Thakur's MZV at $\ww{s}$
\cite{thakur2004function} is defined by 
\[
\zeta_A(\ww{s}) := \sum_{\substack{f_1,\ldots,f_m\in A_+ \\ \deg f_1 > \cdots > \deg f_m}}
\frac{1}{f_1^{s_1}\cdots f_m^{s_m}}\in K_{\infty}.
\]
Carlitz zeta values~\cite{Car35} are special cases of  Thakur's MZVs with depth one. By convention, we put $\zeta_A(\varnothing)=1$.

For any subfield $L$ of $\mathbb{C}_{\infty}$ containing $K$ and integer $w\ge 0$, we put
\[
\mathcal{Z}_{L,w} := \Span_L \{\zeta_A(\ww{s})\mid \wt(\ww{s}) = w\}.
\]
In~\cite{thakur2010shuffle}, Thakur discovered and proved the so-called $q$-shuffle product, which implies that for any integers $w_1,w_2\geq 0$,
\[
\mathcal{Z}_{L,w_1}\cdot \mathcal{Z}_{L,w_2} \subseteq \mathcal{Z}_{L, w_1+w_2}.
\]
It follows that the $L$-vector space spanned by all Thakur's MZVs, denoted by
\[
\mathcal{Z}_L := \sum_{w=0}^{\infty} \mathcal{Z}_{L,w},
\]
forms an $L$-algebra. 

Using the linear independence criterion of Anderson--Brownawell--Papanikolas \cite{ABP04} as well as period interpretation of Thakur's MZVs proved by Anderson-Thakur \cite{AT09}, we have the following stronger form of function field analogue of Goncharov--Zagier's direct sum conjecture.

\begin{thm}[{\textnormal{\cite[Theorem~2.2.1]{Cha14}}}]\label{thm.Goncharov-Zagier conjecture for function field}
    We keep the notation as above. Then the following results hold.
    \begin{enumerate}  
        \item $\mathcal{Z}_{\overline{K}}$ is a graded algebra, i.e.,  $\mathcal{Z}_{\overline{K}} = \bigoplus_{w=0}^{\infty} \mathcal{Z}_{\overline{K}, w}$.
        \item The natural map $\mathcal{Z}_K\otimes_K \overline{K}\to \mathcal{Z}_{\overline{K}}$ is an isomorphism.
    \end{enumerate}
\end{thm}

In other words, $\mathcal{Z}_{\overline{K}}$ is a graded algebra ``defined over $K$'' in the sense that $K$-linear independence of Thakur's MZVs implies $\overline{K}$-linear independence.

Motivated by Gangl--Kaneko--Zagier, MES in positive
characteristic were introduced by Chen~\cite{Chen2017} in the case of rank two and depth two, and
the theory was later extended to MES of arbitrary rank and arbitrary depth
in~\cite{CCHT25}. Hereafter, MES will always refer to multiple Eisenstein
series in positive characteristic whenever it is clear from the context. The primary goal of this paper is to
extend \cref{thm.Goncharov-Zagier conjecture for function field}
to MES. We briefly recall their definition here, and refer to
\S\ref{sec.preliminary} for further details. 

Fix a positive integer $r$. For any $\bff = (f_1,\ldots,f_r) \in A^r$ and $\bfz = (z_1,\ldots,z_r)\in \CC_\infty^r$, we put
\[
    \langle \bff, \bfz \rangle := \bm{f}\bfz^{\tr}=f_1z_1 + \cdots + f_rz_r \in \CC_\infty.
\] By viewing the set $A_{+}$ as an analogue of the set of positive integers, we define a partial order $\succ$ on $A^{r}$ given in \cref{df.partial order on Λ_z}, and then define the following MES of rank $r$, which are rigid analytic functions on the
Drinfeld symmetric space
\[
\Omega^r
=
\left\{
\bfz=(z_1,\ldots,z_r)\in\CC_\infty^r
\mid
z_1,\ldots,z_r
\text{ are $K_\infty$-linearly independent and } z_r=1
\right\}.
\]

\begin{df}[{\textnormal{\cite[Definition~1.4]{CCHT25}}}]\label{df.multiple Eisenstein series}
    Fix a positive integer $r$.
    For any $\mathbf{z}\in \Omega^r$ and nonempty index $\ww{s} = (s_1,\ldots,s_m)$, we define the multiple Eisenstein series of rank $r$ by
    \[
    E_r(\ww{s}; \mathbf{z}) = \sum_{\substack{\bm{f}_1,\ldots,\bm{f}_m \in A^r \\  \bm{f}_1\succ \cdots \succ  \bm{f}_m \succ 0}}\frac{1}{\left\langle \bm{f}_1,\bfz\right\rangle^{s_1}\cdots \left\langle \bm{f}_m, \mathbf{z}\right\rangle^{s_m}}.
    \]
    We adopt the convention that $E_r(\emptyset; \mathbf{z}): = 1$.
\end{df}
Note that the above definition was initiated by Chen~\cite{Chen2017} in the case of $r=2$ and $m=2$. We also mention that, in \cite{Pel2025}, Pellarin studied vector-valued versions of MES in the rank two case.
In analogy with the classical case, our MES can be regarded as a higher-rank deformation of Thakur's MZVs as we have 
\[
E_1(\ww{s}; \mathbf{z}) = \zeta_A(\ww{s})
\]
(see \cite[p.~4]{CCHT25}).
Moreover, for $r\geq 2$, the MES $E_r(\ww{s}; \mathbf{z})$ admits a ``Goss expansion" with constant term $E_{r-1}(\ww{s}; \mathbf{w})$  (see  \cref{prop.Goss expansion of MES}).

Let $r\geq 1$.
For any  subfield $L$ of $\mathbb{C}_{\infty}$ containing $K$ and integer $w\geq 0$, we set
\[
\mathcal{Z}^{(r)}_{L, w} := \Span_{L} \{E_r(\ww{s}; \mathbf{z})\mid \wt(\ww{s}) = w\}
\]
and let
\[
\mathcal{Z}^{(r)}_L := \sum_{w=0}^{\infty} \mathcal{Z}^{(r)}_{L,w}.
\]
It is shown in~\cite[Theorem~1.6]{CCHT25} that Thakur's $q$-shuffle product can be lifted to MES. Therefore, for any integers $w_1,w_2\geq 0$, 
\begin{equation}\label{eq.q-shuffle-product-MES}
    \mathcal{Z}^{(r)}_{L,w_1}\cdot \mathcal{Z}^{(r)}_{L, w_2} \subseteq \mathcal{Z}^{(r)}_{L, w_1+w_2}
\end{equation}
and $\mathcal{Z}^{(r)}_L$ becomes an $L$-algebra.

\subsection{Main results and idea of proof}

The main result of this paper is stated as follows. 

\begin{thm} \label{thm.main-thm}
    For any integer $r\geq 1$, the following results hold.
    \begin{itemize}
        \item[(i)]  $\Zcal^{(r)}_{\overline{K}}$ is a graded algebra, i.e.,  \[\Zcal^{(r)}_{\overline{K}} = \bigoplus_{w=0}^{\infty}\Zcal^{(r)}_{\overline{K},w}.\]
        \item[(ii)] The natural map $\mathcal{Z}^{(r)}_{K}\otimes_K \overline{K}\twoheadrightarrow  \mathcal{Z}^{(r)}_{\overline{K}}$ is an isomorphism of $\overline{K}$-algebras. 
    \end{itemize}
\end{thm}

The proof of our main results is outlined as follows. We proceed by induction
on $r$ to establish \cref{thm.main-thm}, with the base case $r=1$ following
from \cref{thm.Goncharov-Zagier conjecture for function field}. The bridge from MES of rank $r+1$ to MES of rank $r$ is the ``expansion at infinity" in the parameter $t_{\Lambda_{\bfz'}}$.  The key ingredient
of the induction step is to analyze the $t_{\Lambda_{\bfz'}}$-expansion of any
MES of rank $r+1$. We compute its coefficients and show that
each coefficient can be expressed in terms of MES of rank $r$, divided by
a suitable power of the  Drinfeld discriminant form $\Delta_r$ (see \S\ref{sec.t-exp-MES}).

The explicit description mentioned above allows us to determine the ``weights'' of the
coefficients appearing in the $t_{\Lambda_{\bfz'}}$-expansion. Consequently,
any linear relation among MES of rank $r+1$ induces, coefficient by
coefficient, corresponding relations among MES of rank $r$. Using the
uniqueness of the $t_{\Lambda_{\bfz'}}$-expansion together with the induction
hypothesis for MES of rank $r$, we then obtain the desired conclusion for
MES of rank $r+1$, and hence establish \cref{thm.main-thm} (see \S\ref{sec.proof of 1.5} for details).

\begin{rmk}
    After completing the proofs of our main results, we learned that Bachmann and Kanno had posted on arXiv a paper proving \cref{Thm:BK26}. We then observed that their method could be adapted to strengthen our \cref{thm.main-thm} by replacing $\ov{K}$ with $\CC_\infty$ when $r\geq 2$. Precisely, the proof of the first assertion in arbitrary rank and that of the second assertion in rank two follow a strategy similar to that of \cite{BK26}, while the second assertion in arbitrary rank is obtained by combining the rank two case with the induction argument used in the proof of \cref{thm.main-thm}. We refer to Appendix~\ref{appendix} for details.
\end{rmk} 

\section{Preliminaries}\label{sec.preliminary}

In this section, we recall the necessary preliminaries on Drinfeld modular forms, multiple Eisenstein series, and multiple Goss sums, which will be used primarily in the proof of \cref{thm.main-thm} in \S\ref{sec.proof of 1.5}.
\subsection{The Drinfeld symmetric space \texorpdfstring{$\Omega^r$}{Ω\^ r} }

Given a positive integer $r\ge 1$, we identify the Drinfeld symmetric space $\Omega^r$ of rank $r$ with the following subset of $\CC_\infty^r$:
\[
\Omega^r = \{\bfz = (z_1,\ldots,z_r) \in \CC_\infty^r \mid z_1,\ldots,z_r \text{ are $K_\infty$-linearly independent and } z_r=1\}.
\]
It is known from~\cite[Proposition~6.1]{Drinfeld1974} that $\Omega^r$ admits a rigid analytic structure and we denote by $\Ocal(\Omega^r)$ the $\CC_{\infty}$-algebra of rigid analytic functions on $\Omega^r$.

Then for any $\bfz=(z_{1},\ldots,z_{r-1},1)\in \Omega^r$ and $\bff=(f_{1},\ldots,f_{r})\in A^{r}$, we denote by $\langle \bff,\bfz \rangle:=f_{1}z_{1}+\cdots+f_{r-1}z_{r-1}+f_{r}\in \CC_{\infty}$. We put
\[
\Lambda_{\bfz} := \{\langle \bff, \bfz \rangle \mid \bff\in A^r\} \subset \CC_\infty,
\] and note that $\Lambda_{\bfz}$
forms an $A$-lattice of rank $r$ inside $\CC_\infty$.

We now recall the notion of the $t_{\Lamz}$-expansion.
Assume $r\ge 2$.
For $\bfz = (z_1,\ldots,z_{r-1},1)\in \Omega^r$, we put $\bfz' = (z_2,\ldots,z_{r-1},1) \in \Omega^{r-1}$ and consider the function
\begin{equation}\label{eq.t-Lambda}
    t_{\Lamz}(z_1) := \sum_{\lambda \in \Lamz} \frac{1}{z_1 + \lambda} = \exp_{\Lamz}(z_1)^{-1}, 
\end{equation}
where 
\[
\exp_{\Lamz}(z_1) := z_1\prod_{0\neq \lambda \in \Lamz} \left(1-\frac{z_1}{\lambda}\right) 
\]
is the exponential function associated with the rank $r-1$ lattice $\Lamz$ (see \cite[\S4]{goss1996basic}).

Let 
\[
\Gamma_r := \left\{\left(\begin{array}{c|c}
    1 & * \\
    \hline
    0 & I_{r-1}
\end{array}\right)\right\} \subset \GL_r(A)
\]
and denote the subspace of all $\Gamma_r$-invariant functions by $\Ocal(\Omega^r)^{\Gamma_r}$. 
By \cite[Proposition~5.4]{BBP24} (see also \cite{Gek25-DMF-VII}), every $F\in \Ocal(\Omega^r)^{\Gamma_r}$ admits a unique $t_{\Lambda_{\bfz'}}$-expansion:
\[
F(\mathbf{z})=\sum_{n=-\infty}^{\infty}F_n(\mathbf{z}')t_{\Lambda_{\mathbf{z}'}}(z_1)^n,
\]
valid for $\mathbf{z} = (z_1, \ldots, z_{r-1}, 1)\in \Omega^r$ in a neighborhood of infinity.
We further define the order of $F$ at infinity by $\ord_r (F) := \inf \{n \in \ZZ \mid F_n(\bfz') \neq 0 \}\in \mathbb{Z}\cup \{\pm \infty\}$.

\subsection{Multiple Eisenstein series and multiple Goss sums}

To define MES, we first recall the partial order
introduced in \cite{Chen2017,CCHT25}, reformulated here as a partial order
on $A^r$.
\begin{df}\label{df.partial order on Λ_z}
    For every nonzero
    $\bm{f}=(f_1,\ldots,f_r)\in A^r$, we put
    \[
    \ell(\bm{f})
    :=
    \min\{i\in\{1,\ldots,r\}\mid f_i\neq0\}.
    \]
    We define a partial order $\succ$ on $A^r$ as follows.
    For nonzero $\bm{f}\in A^r$, we write $\bm{f}\succ0$ if $f_{\ell(\bm{f})}\in A_+$.
    For $\bm{f},\bm{g}\succ0$, we write
    $\bm{f}\succ\bm{g}$
    if either
    $\ell(\bm{f})<\ell(\bm{g})$,
    or
    \[
    \ell(\bm{f})=\ell(\bm{g})
    \qquad\text{and}\qquad
    \deg f_{\ell(\bm{f})}
    >
    \deg g_{\ell(\bm{g})}.
    \]
\end{df}
 For any $r\geq 1$ and any nonempty index $\ww{s}=(s_1,\ldots,s_m)$, recall that the associated MES of rank $r$ is defined by 
\[
E_r(\ww{s}; \mathbf{z}) = \sum_{\substack{\bm{f}_1,\ldots,\bm{f}_m \in A^r \\  \bm{f}_1\succ \cdots \succ  \bm{f}_m \succ 0}}\frac{1}{\left\langle \bm{f}_1,\bfz\right\rangle^{s_1}\cdots \left\langle \bm{f}_m, \mathbf{z}\right\rangle^{s_m}}.
\]
It was shown in~\cite{CCHT25} that $E_r(\ww{s};\bfz)$ is a rigid analytic function on $\Omega^r$. Moreover, for $r\geq 2$, it is invariant under $\Gamma_r$, and therefore admits a unique $t_{\Lamz}$-expansion.

Given $r\ge 1$ and an $A$-lattice $\Lambda$ of rank $r$ in $\CC_\infty$, we write 
\begin{equation}\label{eq.exponential coeff e_n}
    \exp_{\Lambda}(X):=X\prod_{0\neq \lambda \in \Lambda} \left(1-\frac{X}{\lambda}\right)  = \sum_{n=0}^{\infty} e_n(\Lambda)X^{q^n}.
\end{equation}
For $\bfz\in\Omega^r$, we also put $e_n(\bfz):=e_n(\Lambda_\bfz)$. 
By \cite[Proposition~15.3]{BBP24}, the function $\mathbf{z}\mapsto e_n(\mathbf{z}): \Omega^r\to \mathbb{C}_{\infty}$ is a rigid analytic function on $\Omega^r$.
Hence, we may regard $\exp_{\Lambda_{\mathbf{z}}}(X)$ as a power series with coefficients in $\mathcal{O}(\Omega^r)$.

With the necessary setup, we recall the notion of Goss polynomials:
\begin{df}[Goss polynomials, \cite{Gos1980, Gekeler88}]\label{df.Goss polynomial}
    Fix $r\ge 1$ and an $A$-lattice $\Lambda$ of rank $r$ in $\mathbb{C}_{\infty}$.
    We define the Goss polynomials by the sequence of polynomials $\{G_k^{\Lambda}(X)\}_{k=1}^{\infty}$ in $\mathbb{C}_{\infty}[X]$ satisfying the following recurrence:
    \begin{enumerate}
        \item $G_1^{\Lambda}(X) = X$, and
        \item for an integer $k > 1$, 
        \[
        G_k^{\Lambda}(X) = X\sum_{q^i<k}e_i(\Lambda)G_{k-q^i}^{\Lambda}(X).
        \]
    \end{enumerate}
\end{df}

Since $\mathbf{z}\mapsto e_n(\mathbf{z})$ is a rigid analytic function on $\Omega^r$, the coefficients of $G^{\Lambda_{\mathbf{z}}}_k(X)$ are also rigid analytic functions on $\Omega^r$.
In addition, if $\exp_{\Lambda}(X) \in K[\![X]\!]$, then $G^{\Lambda}_k(X)\in K[X]$ for every integer $k\geq 1$.

The core property of Goss polynomials is that they express the reciprocal power sums over lattices in $\mathbb{C}_\infty$ as polynomials in $t_{\Lambda_{\mathbf z}}(X)$, where $t_{\Lambda_{\mathbf z}}(X)$ is obtained from \eqref{eq.t-Lambda} by substituting $\bfz'\rightarrow\bfz$ and $z_1\rightarrow X$.
\begin{prop}[\cite{Gos1980, Gekeler88}]\label{prop.G_a(t(X)) = sum 1/(X+λ)^a}
    For any integer $k\geq 1$, the following properties hold.
    \begin{enumerate}
        \item $G_k^{\Lambda_{\mathbf{z}}}(t_{\Lambda_{\bfz}}(X)) = \sum_{\lambda\in \Lambda_{\mathbf{z}}}\frac{1}{(X+\lambda)^k}$.
        \item $G_k^{\Lambda_{\mathbf{z}}}(X)$ is a monic polynomial of degree $k$.
        \item $X\mid G_k^{\Lambda_{\mathbf{z}}}(X)$.
        \item if $1\le k\le q$, $G_k^{\Lambda_{\mathbf{z}}}(X) = X^k$.
    \end{enumerate}
\end{prop}

Using Goss polynomials, we have the notion of multiple Goss sums (abbreviated as MGS) introduced in~\cite{CCHT25}.
\begin{df}[Multiple Goss sums]\label{df.multiple Goss sums}
    For any $r\geq 2$ and nonempty index $\ww{s} = (s_1,\ldots,s_m)$, we define the multiple Goss sum of rank $r$ by 
    \[
    G_r(\ww{s}; \mathbf{z}) = \sum_{\substack{f_1,\ldots,f_m\in A_+ \\ \deg f_1 > \cdots > \deg f_m}}G_{s_1}^{\Lamz}(t_{\Lamz}(f_1z_1))\cdots G_{s_m}^{\Lamz}(t_{\Lamz}(f_mz_1)).
    \]
    By convention, we put $G_r(\varnothing;\bfz)=1$.
\end{df}

We also recall the notations $\ww{s}^{(i)}$ and $\ww{s}_{(i)}$.
The former one is the index obtained from $\ww{s}$ by deleting the first $i$ entries and the latter one is the index obtained from $\ww{s}$ by retaining the first $i$ entries.
The precise definitions are given as follows.
\begin{df}\label{df.a^(i) and a_(i)}
    For a nonempty index $\ww{s}=(s_1,\ldots,s_m)$ and $i\geq0$, we let
    \[
    \ww{s}^{(i)}
    =
    \begin{cases}
    \ww{s}, & i=0,\\
    (s_{i+1},\ldots,s_m), & 1\leq i<m,\\
    \emptyset, & i\geq m,
    \end{cases}
    \qquad\text{and}\qquad
    \ww{s}_{(i)}
    =
    \begin{cases}
    \emptyset, & i=0,\\
    (s_1,\ldots,s_i), & 1\leq i<m,\\
    \ww{s}, & i\geq m.
    \end{cases}
    \]
\end{df}

Using \cref{prop.G_a(t(X)) = sum 1/(X+λ)^a} and \cref{df.multiple Eisenstein series}, we obtain the Goss expansion of MES:
\begin{prop}[{\cite[Proposition 2.3]{CCHT25}}]\label{prop.Goss expansion of MES}
    For any integer $r\ge 2$ and any index $\ww{s}$, we have 
    \[
    E_r(\ww{s}; \mathbf{z}) = E_{r-1}(\ww{s}; \mathbf{z}') + \sum_{i=1}^{\dep(\ww{s})}E_{r-1}(\ww{s}^{(i)}; \mathbf{z}')G_r(\ww{s}_{(i)}; \mathbf{z}).
    \]
\end{prop}

The authors of \cite{CCHT26} computed the $t_{\Lambda_{\mathbf{z}'}}$-expansion of $t_{\Lambda_{\mathbf{z}'}}(f_{1}z_{1})$ and  established the following lower bound (see the proof of \cite[Lemma~2.9]{CCHT26}):
\begin{equation}\label{eq.ord_r(t(fz))}
    \ord_r(t_{\Lamz}(fz_1)) \ge q^{(r-1)\deg f}.
\end{equation}
Furthermore, by using \eqref{eq.ord_r(t(fz))}, it is shown in~\cite[Proposition 2.10]{CCHT26} that the constant term of the $t_{\Lambda_{\bfz'}}$-expansion of $E_r(\ww{s}; \mathbf{z})$ is $E_{r-1}(\ww{s}; \mathbf{z}')$.

\subsection{Drinfeld modular forms}
Throughout this paper, the term ``Drinfeld modular form of rank $r$'' refers
to a Drinfeld modular form of type $0$ for $\GL_r(A)$. For any integer $r\geq 2$ and integer $w\geq 0$, we denote by
$\mathscr{M}^{(r)}_w$
the $\CC_\infty$-vector space of Drinfeld modular forms of rank $r$ and
weight $w$, and by $\mathscr{M}^{(r)}$ the $\CC_\infty$-algebra of Drinfeld modular forms of rank $r$. Then 
\[
\mathscr{M}^{(r)}
=
\bigoplus_{w\geq 0}\mathscr{M}^{(r)}_w
\]
forms a graded $\CC_\infty$-algebra (for the definition of Drinfeld modular forms and further details, see \cite{BBP24}).

Now let $r\geq1$ and let $\mathbb{C}_{\infty}[\tau]$ denote the twisted polynomial ring
determined by the relation
$\tau c=c^q\tau$ for all $c\in\CC_\infty$. By the Drinfeld uniformization theorem
(see \cite[\S4.3]{goss1996basic}),
for each $\mathbf{z}\in\Omega^r$, the $A$-lattice
$\Lambda_{\mathbf{z}}$ determines a rank $r$ Drinfeld $A$-module
\[
\phi^{\Lambda_{\mathbf{z}}}:
A\longrightarrow\mathbb{C}_{\infty}[\tau],
\qquad
a\longmapsto\phi_a^{\Lambda_{\mathbf{z}}},
\]
characterized by
\[
\exp_{\Lambda_{\mathbf{z}}}(aX)
=
\phi_a^{\Lambda_{\mathbf{z}}}
(\exp_{\Lambda_{\mathbf{z}}}(X))
\]
for all $a\in A$. In particular,
\begin{equation}\label{eq.exp linearity}
\exp_{\Lambda_{\mathbf{z}}}(\theta X)
=
\phi_{\theta}^{\Lambda_{\mathbf{z}}}
(\exp_{\Lambda_{\mathbf{z}}}(X)).
\end{equation}
Now, write 
\[
\phi^{\Lambda_{\mathbf{z}}}_{\theta} = \theta + g_1(\mathbf{z})\tau + \cdots + g_{r-1}(\mathbf{z})\tau^{r-1} + \Delta_r(\mathbf{z})\tau^r.
\]
Notice that $\Delta_r(\bfz)$ is nowhere vanishing. 

In analogy with the classical setting, where the coefficients of the Weierstrass equation of an elliptic curve associated with a lattice generate the ring of modular forms for $\SL_2(\ZZ)$, we have the following results due to \cite{BBP24}.
\begin{prop}\label{prop.ring of modular forms}
    For any integer $r\geq2$, the following assertions hold.
    \begin{enumerate}
        \item (\cite[Proposition~15.12 and Proposition~16.3]{BBP24}) For $i=1,\ldots,r-1$, $g_i$ is a Drinfeld modular form of
        weight $q^i-1$, and $\Delta_r$ is a Drinfeld cusp form of weight
        $q^r-1$.
    
        \item (\cite[Theorem~17.5(a)]{BBP24}) We have
        \[
        \mathscr{M}^{(r)}
        =
        \CC_\infty[g_1,\ldots,g_{r-1},\Delta_r],
        \]
        and $g_1,\ldots,g_{r-1},\Delta_r$ are algebraically independent
        over $\CC_\infty$.
    
        \item (\cite[Theorem~17.5(a)]{BBP24}) We have
        \[
        \mathscr{M}^{(r)}
        =
        \CC_\infty[
            E_{q-1}^{(r)},\ldots,E_{q^r-1}^{(r)}
        ],
        \]
        and $E_{q-1}^{(r)},\ldots,E_{q^r-1}^{(r)}$ are algebraically
        independent over $\CC_\infty$, where, for $s\geq 1$ with
        $q-1\mid s$,
        \begin{equation}\label{eq.single Eisenstein series of weight a}
        E_s^{(r)}(\bfz)
        =
        \sum_{0\neq\lambda\in\Lambda_{\bfz}}
        \frac{1}{\lambda^s}
        \end{equation}
        denotes the single Eisenstein series of rank $r$ and weight $s$.
    
        \item (\cite[Proposition~15.3]{BBP24}) For every $n\geq0$, $e_n(\bfz)$ is a Drinfeld modular form of weight $q^n-1$. 
    
        \item (\cite[(15.1) and~(15.6)]{BBP24})  We have
        \[
        K[g_1,\ldots,g_{r-1},\Delta_r]
        =K[e_i\mid i\geq 1]=
        K[E_{q-1}^{(r)},\ldots,E_{q^r-1}^{(r)}].
        \]
    \end{enumerate}
\end{prop}

For any integer $r\geq 2$ and any subfield $L$ of $\CC_\infty$ containing $K$, we put
\[
\mathscr{M}^{(r)}_L
:=
L[g_1,\ldots,g_{r-1},\Delta_r]
\subseteq
\mathscr{M}^{(r)},
\]
and denote by $\mathscr{M}^{(r)}_{L,w}$ its weight $w$ homogeneous
component. The following lemma shows that every element of
$\mathscr{M}^{(r)}_{L,w}$ can be expressed as an $L$-linear combination
of MES of weight $w$.

\begin{lem}\label{lem:DMF-is-MES}
    For any integer $r\geq 2$ and any subfield $L$ of $\CC_\infty$ containing $K$, we have
    \[
    \mathscr{M}^{(r)}_L=L[g_1,\ldots,g_{r-1},\Delta_r]
    \subset
    \Zcal_{L}^{(r)}.
    \]
    Moreover, if
    $F\in L[g_1,\ldots,g_{r-1},\Delta_r]$
    is a Drinfeld modular form of weight $w$, then
    \[
    F\in \Zcal_{L,w}^{(r)}.
    \]
\end{lem}
\begin{proof}
    By \eqref{eq.single Eisenstein series of weight a}, for every
    $s\geq1$ with $q-1\mid s$, one checks that
    \[
    E_s^{(r)}(\bfz)=-E_r(s;\bfz),
    \]
    and hence
    \[
    E_s^{(r)}\in\Zcal_{L,s}^{(r)}.
    \]
    
    By \cref{prop.ring of modular forms}, we have an equality of graded
    $L$-algebras
    \[
    L[g_1,\ldots,g_{r-1},\Delta_r]
    =
    L[E_{q-1}^{(r)},\ldots,E_{q^r-1}^{(r)}].
    \]
    Thus, if $F\in\mathscr{M}^{(r)}_{L,w}$, then $F$ is an $L$-linear
    combination of monomials in
    $E_{q-1}^{(r)},\ldots,E_{q^r-1}^{(r)}$ of total weight $w$.
    Since each $E_{q^i-1}^{(r)}$ belongs to
    $\Zcal_{L,q^i-1}^{(r)}$, it follows from \eqref{eq.q-shuffle-product-MES} that
    \[
    F\in\Zcal_{L,w}^{(r)}.
    \]
    The first assertion follows by summing over all weights.
\end{proof}

\begin{rmk}\label{rmk:rank 1}
We fix a $(q-1)$-st root of $(-\theta)$ and let
\[
\wtd{\pi}
=
(-\theta)^{\frac{q}{q-1}}
\prod_{i=1}^{\infty}
\left(1-\frac{\theta}{\theta^{q^i}}\right)^{-1}
\in
K_{\infty}((-\theta)^{\frac{1}{q-1}})
\]
be a fixed fundamental period of the Carlitz $A$-module, namely,
$\phi_{\theta}^{\Lambda_{\wtd{\pi}A}}=\theta+\tau$
(see \cite[\S3.2]{goss1996basic}). When $r=1$, we have the following.
\begin{enumerate}
    \item The ``rank one discriminant form'' $\Delta_1$ is given by
    $\Delta_1=\wtd{\pi}^{q-1}$ since the Drinfeld $A$-module $\phi^A$
    associated with the point $1\in\Omega^1$ is determined by
    \[
    \phi^A_\theta=\theta+\wtd{\pi}^{q-1}\tau.
    \]
    By the Euler--Carlitz formula \cite{Car35} (see also \cite[Theorem~5.2.1]{thakur2004function}),
    $\Delta_1$ is a $K$-multiple of $\zeta_A(q-1)$. Hence
    \[
    \Delta_1^m\in\Zcal^{(1)}_{K,m(q-1)}
    \qquad (m\geq0),\qquad \text{and}\qquad
    L[\Delta_1]\subseteq\Zcal^{(1)}_L
    \]
    for any subfield $L$ of $\CC_\infty$ containing $K$.

    \item For every $n\geq0$, the Carlitz exponential formula
    (see \cite[\S3.2]{goss1996basic}) gives
    \[
    e_n(1)
    =
    \frac{\wtd{\pi}^{q^n-1}}{D_n}
    =
    \frac{1}{D_n}\Delta_1^{(q^n-1)/(q-1)}
    \in
    \Zcal^{(1)}_{K,q^n-1},
    \]
    where $D_0=1$ and $D_n = \prod_{j=0}^{n-1} (\theta^{q^n} - \theta^{q^j})$.
\end{enumerate}
\end{rmk}

\section{Proof of \texorpdfstring{\cref{thm.main-thm}}{Theorem 1.5}}\label{sec.proof of 1.5}

In this section, we prove \cref{thm.main-thm}.
In \S\ref{sec.t-exp-MES}, we first compute the coefficients in the $t_{\Lamz}$-expansion of MES.
Our main goal is to determine the precise ``weight'' carried by each coefficient in the $t_{\Lamz}$-expansion of MES (see \cref{thm:t-expansion-of-MES}).
This enables us to prove \cref{thm.main-thm}~(i) and (ii) in \S\ref{sec.proof of 1.5-(i)} and \S\ref{sec.proof of 1.5(ii)}, respectively.

\subsection{\texorpdfstring{$t$}{t}-expansion coefficients of multiple Eisenstein series}\label{sec.t-exp-MES}

Fix $r\geq1$ and a subfield $L$ of $\CC_\infty$ containing $K$.
Throughout this section, we always write
\[
\bfz=(z_1,\bfz')\in\Omega^{r+1},
\qquad
\bfz'=(z_2,\ldots,z_r,1)\in\Omega^r.
\]
We study the coefficients occurring in the
$t_{\Lambda_{\bfz'}}$-expansion of MES of rank $r+1$. Note that the natural projection
\[
\Omega^{r+1}\longrightarrow\Omega^r,
\qquad
(z_1,\bfz')\longmapsto\bfz',
\]
induces an injective $\CC_\infty$-algebra homomorphism
\[
\Ocal(\Omega^r)\hookrightarrow\Ocal(\Omega^{r+1}),
\]
and we identify $\Ocal(\Omega^r)$ with its image under this map. 

We first consider the localization
\[
\Zcal_{L}^{(r)}[\Delta_r^{-1}] = \sum_{w\in \ZZ} \Lcal_{L,w}^{(r)},
\]
where
\[
\Lcal_{L,w}^{(r)} := \left\{ \frac{F}{\Delta_r^M} \;\middle|\; M\ge 0 \text{ and } F\in \Zcal_{L, w + M(q^r-1)}^{(r)} \text{ with }w+M(q^r-1)\geq 0 \right\}.
\]
Notice that by \eqref{eq.q-shuffle-product-MES}, for any integers $w_1,w_2$,
\[
\Lcal^{(r)}_{L,w_1}\cdot\Lcal^{(r)}_{L,w_2}\subseteq \Lcal^{(r)}_{L,w_1+w_2}.
\]

\begin{lem}\label{lem:coefficient-form-f}
    Fix an integer $r\ge 1$ and denote by $\bfw$ an arbitrary point varying in $\Omega^r$.
    For $f\in A_+$ with $\deg f=d$, write
    \[
    \phi_f^{\Lambda_{\mathbf{w}}}(X)
    =
    \sum_{i=0}^{rd}g_{f,i}(\mathbf{w})X^{q^i}
    \in\Ocal(\Omega^r)[X].
    \]
    Then
    \[
    g_{f,i}\in K[g_1,\ldots,g_{r-1},\Delta_r]
    \]
    is of weight $q^i-1$ for every $0\leq i\leq rd$. Moreover,
    \[
    g_{f,rd}
    =\Delta_r^{(q^{rd}-1)/(q^r-1)}.
    \]
\end{lem}
\begin{proof}
For $r=1$, the assertions follow from \cref{rmk:rank 1} and the fact that
$\phi^A:A\to\CC_\infty[\tau]$ is an $\FF_q$-algebra homomorphism. For $r\geq 2$, see \cite[Propositions~15.12 and~16.4]{BBP24}.
\end{proof}

\begin{lem}\label{lem.t(fz)-coefficient}
    Fix an integer $r\ge 1$ and denote by $\bfz = (z_1, \bfz')$ an arbitrary point varying in $\Omega^{r+1}$.
    For $f\in A_+$, write
    \[
    t_{\Lamz}(fz_1)=\sum_{N\geq 0}\Acal_{f,N}(\bfz')\cdot t_{\Lamz}(z_1)^N.
    \]
    Then 
    \[
    \Acal_{f,N}(\bfz')\in\Lcal_{K,1-N}^{(r)}
    \]
    for all $N\geq 0$.
\end{lem}
\begin{proof}
    Let $d=\deg f$. 
    We note that from~\eqref{eq.t-Lambda} and~\eqref{eq.exp linearity}, we have
    \begin{align*}
        t_{\Lamz}(fz_1)
        &= \frac{1}{\exp_{\Lamz}(fz_1)}
        = \frac{1}{\phi_f^{\Lamz}\left(t_{\Lamz}(z_1)^{-1}\right)} \\
        &= \frac{1}{\displaystyle\sum_{i=0}^{rd} g_{f,i}(\bfz')t_{\Lamz}(z_1)^{-q^i}} \\
        &= \frac{t_{\Lamz}(z_1)^{q^{rd}}}{g_{f,rd}(\bfz')} \left( 1+ \sum_{i=0}^{rd-1} \frac{g_{f,i}(\bfz')}{g_{f,rd}(\bfz')} t_{\Lamz}(z_1)^{q^{rd}-q^i} \right)^{-1}.
    \end{align*}
    By \cref{lem:coefficient-form-f}, together with \cref{rmk:rank 1} for $r=1$ and \cref{lem:DMF-is-MES} for $r\geq2$, we have
    \[
    g_{f,i}(\bfz')\in\Zcal_{K,q^i-1}^{(r)},\qquad
    g_{f,rd}(\bfz')^{-1}\in\Lcal_{K,1-q^{rd}}^{(r)},
    \]
    and hence
    \[
    \frac{g_{f,i}(\bfz')}{g_{f,rd}(\bfz')}
    \in
    \Lcal^{(r)}_{K,q^i-q^{rd}}
    \qquad
    (0\leq i\leq rd-1).
    \]

    Expanding the inverse as a formal power series gives
    \[
    \left(
    1+
    \sum_{i=0}^{rd-1}
    \frac{g_{f,i}(\bfz')}{g_{f,rd}(\bfz')}
    t_{\Lamz}(z_1)^{q^{rd}-q^i}
    \right)^{-1}
    =
    \sum_{m\geq0}
    (-1)^m
    \left(
    \sum_{i=0}^{rd-1}
    \frac{g_{f,i}(\bfz')}{g_{f,rd}(\bfz')}
    t_{\Lamz}(z_1)^{q^{rd}-q^i}
    \right)^m.
    \]
    Consider a monomial occurring on the right-hand side,
    \[
    \prod_{j=1}^m
    \frac{g_{f,i_j}(\bfz')}{g_{f,rd}(\bfz')}\,
    t_{\Lamz}(z_1)^{q^{rd}-q^{i_j}}.
    \]
    Its coefficient belongs to $\Lcal_{K, \sum_{j=1}^m(q^{i_j}-q^{rd})}^{(r)},$ while its power of $t_{\Lamz}(z_1)$ is
    \[
    \sum_{j=1}^m(q^{rd}-q^{i_j}).
    \]
    Thus, if such a monomial contributes to the coefficient of $t^N$
    in $t_{\Lamz}(fz_1)$, then
    \[
    N = q^{rd} + \sum_{j=1}^m(q^{rd}-q^{i_j}),
    \]
    and its coefficient belongs to
    \begin{align*}
        \Lcal^{(r)}_{K, 1-q^{rd}
        +\sum_{j=1}^m(q^{i_j}-q^{rd})}
        &=
        \Lcal_{K,1-N}^{(r)}.
    \end{align*}
    Therefore every term contributing to $\Acal_{f,N}(\bfz')$ lies in
    $\Lcal^{(r)}_{K, 1-N}$, and hence
    \[
    \Acal_{f,N}(\bfz')\in\Lcal_{K, 1-N}^{(r)}
    \]
    for all $N\geq0$.
\end{proof}

\begin{lem}\label{lem.Goss-poly-coefficient}
    Fix an integer $r\ge 1$ and denote by $\bfw$ an arbitrary point varying in $\Omega^r$.
    For $s\geq 1$, write
    \[
    G_s^{\Lambda_{\mathbf{w}}}(X)=\sum_{i=1}^s \Bcal_{s,i}(\mathbf{w})X^{i}\in\Ocal(\Omega^r)[X].
    \]
    Then 
    \[
    \Bcal_{s,i}\in\Lcal_{K, s-i}^{(r)}
    \]
    for every $1\leq i\leq s$.
\end{lem}
\begin{proof}
    Recall that the coefficient $e_n(\mathbf{w}) \in \Ocal(\Omega^r)$ of
    $\exp_{\Lambda_{\mathbf{w}}}(X)$ is defined in \eqref{eq.exponential coeff e_n}. By \cref{rmk:rank 1} for $r=1$ and \cref{prop.ring of modular forms,lem:DMF-is-MES} for $r\geq2$, we have
\[
e_n(\mathbf{w})\in\Zcal^{(r)}_{K,q^n-1}.
\]
    According to \cite[(3.8)]{Gekeler88}, the Goss polynomial
    associated with $\Lambda_{\mathbf{w}}$ is given by
    \[
    G_s^{\Lambda_{\mathbf{w}}}(X)
    =
    \sum_{j=0}^{s-1}\sum_{\underline{i}}
    \binom{j}{\underline{i}}
    \bm{e}^{\underline{i}}(\bfw)X^{j+1},
    \]
    where the inner sum runs over all tuples
    $\underline{i}=(i_0,\ldots,i_{s-1})$ of non-negative integers satisfying
    \[
    \sum_{l=0}^{s-1} i_l=j
    \qquad\text{and}\qquad
    \sum_{l=0}^{s-1} i_lq^l=s-1,
    \]
    and
    \[
    \bm{e}^{\underline{i}}(\bfw)
    :=
    e_0(\bfw)^{i_0}\cdots e_{s-1}(\bfw)^{i_{s-1}}.
    \]
    Hence, by \eqref{eq.q-shuffle-product-MES},
    \[
    \bm{e}^{\underline{i}}(\bfw)
    \in
    \Zcal^{(r)}_{K,\sum_{l=0}^{s-1} i_l(q^l-1)}
    =
    \Zcal^{(r)}_{K,s-1-j}.
    \]
    Taking $j=i-1$, we obtain
    \[
    \Bcal_{s,i}(\bfw)
    \in
    \Zcal^{(r)}_{K,s-i}
    \subseteq
    \Lcal^{(r)}_{K,s-i}.
    \]
\end{proof}
    
\begin{lem}\label{lem:t-expansion of G(t(fz))}
    Fix an integer $r\ge 1$ and denote by $\bfz = (z_1, \bfz')$ an arbitrary point varying in $\Omega^{r+1}$.
    For $f\in A_+$ and $s\geq 1$, write
    \[
    G_s^{\Lamz}(t_{\Lamz}(fz_1))=\sum_{N\geq 0}\Ccal_{s,f,N}(\bfz')\cdot t_{\Lamz}(z_1)^{N}.
    \]
    Then 
    \[
    \Ccal_{s,f,N}(\bfz')\in\Lcal_{K, s-N}^{(r)}
    \]
    for all $N\geq 0$.
\end{lem}
\begin{proof}
    Using the same notation as in \cref{lem.t(fz)-coefficient} and \cref{lem.Goss-poly-coefficient}, we see that
    \[
    G_s^{\Lamz}(t_{\Lamz}(fz_1)) = \sum_{i=1}^s \Bcal_{s,i}(\bfz')\left(\sum_{M\geq 0}\Acal_{f,M}(\bfz')\cdot t_{\Lamz}(z_1)^M\right)^i
    \]
    Comparing the coefficient of $t_{\Lamz}(z_1)^N$, we obtain
    \begin{align*}
        \Ccal_{s,f,N}(\bfz') &= \sum_{i=1}^s \Bcal_{s,i}(\bfz') \left(\sum_{M_1 +\cdots + M_i = N}\prod_{j=1}^i  \Acal_{f,M_j}(\bfz')\right) \\
        &= \sum_{i=1}^s \sum_{M_1 +\cdots + M_i = N} \Bcal_{s,i}(\bfz')\prod_{j=1}^i \Acal_{f,M_j}(\bfz').
    \end{align*}
    By \cref{lem.t(fz)-coefficient} and \cref{lem.Goss-poly-coefficient}, for each $1\le i \le s$, we have
    \[
    \Bcal_{s,i}(\bfz')\prod_{j=1}^i  \Acal_{f,M_j}(\bfz') \in \Lcal_{K, s-i}^{(r)} \cdot \Lcal_{K, 1-M_1}^{(r)}\cdots \Lcal_{K, 1-M_i}^{(r)} \subset\Lcal_{K, s-N}^{(r)}. 
    \]
    Therefore, $\Ccal_{s,f,N}(\bfz') \in \Lcal_{K,s-N}^{(r)}$, as desired.
\end{proof}

\begin{prop}\label{prop:t-expansion of MGS}
    Fix an integer $r\ge 1$ and denote by $\bfz = (z_1, \bfz')$ an arbitrary point varying in $\Omega^{r+1}$.
    For any index $\ww{s}$, write
    \[
    G_{r+1}(\ww{s};\bfz)=\sum_{N\geq 0}\alpha_{\ww{s},N}(\bfz')\cdot t_{\Lamz}(z_1)^{N}.
    \]
    Then 
    \[
    \alpha_{\ww{s},N}(\bfz')\in\Lcal_{K, \wt(\ww{s})-N}^{(r)}
    \]
    for all $N\geq 0$.
\end{prop}
\begin{proof}
    We may assume that $\ww{s}=(s_1,\ldots,s_m)$ is nonempty. Fix $f_1,\ldots,f_m\in A_+$ with $\deg f_1 > \cdots > \deg f_m$.
    Then by \cref{lem:t-expansion of G(t(fz))}, we have 
    \begin{align*}
        &G_{s_1}^{\Lamz}(t_{\Lamz}(f_1z_1))\cdots G_{s_m}^{\Lamz}(t_{\Lamz}(f_mz_1)) \\
        &= \left(\sum_{N_1=0}^{\infty}\Ccal_{s_1,f_1,N_1}(\bfz')t_{\Lamz}(z_1)^{N_1}\right)\cdots \left(\sum_{N_m=0}^{\infty}\Ccal_{s_m,f_m,N_m}(\bfz')t_{\Lamz}(z_1)^{N_m}\right) \\
        &= \sum_{N=0}^{\infty}\left(\sum_{\substack{N_1+\cdots+N_m=N \\ N_1,\ldots,N_m\ge 0}}\Ccal_{s_1,f_1,N_1}(\bfz')\cdots \Ccal_{s_m, f_m, N_m}(\bfz')\right)t_{\Lamz}(z_1)^N
    \end{align*}
    and 
    \[
    \Ccal_{s_1,f_1,N_1}(\bfz')\cdots \Ccal_{s_m,f_m,N_m}(\bfz')\in \Lcal_{K, s_1-N_1}^{(r)}\cdots \Lcal_{K, s_m-N_m}^{(r)} \subseteq \Lcal_{K, \wt(\ww{s})-N}^{(r)}.
    \]
    
    Now, note that by~\eqref{eq.ord_r(t(fz))}, for all $s\geq 1$ and $f\in A_+$, we have
    \[
    \ord_{r+1}(G_s(t_{\Lamz}(fz_1))) \ge q^{r\deg f}.
    \]
    In particular, if $\deg f > \frac{1}{r}\log_q (N+1)$, then $\ord_{r+1}(G_s(t_{\Lamz}(fz_1))) > N$.
    Hence, for all $N\geq 0$, 
    \begin{multline*}
        G_{r+1}(\ww{s}; \mathbf{z}) \\
        \equiv \sum_{\substack{f_1,\ldots,f_m\in A_+ \\ \frac{1}{r}\log_q(N+1)\ge \deg f_1 > \cdots > \deg f_m}}G_{s_1}^{\Lamz}(t_{\Lamz}(f_1z_1))\cdots G_{s_m}^{\Lamz}(t_{\Lamz}(f_mz_1)) 
        \pmod{t_{\Lamz}(z_1)^{N+1}}.
    \end{multline*}
    Therefore, we obtain
    \[
    \alpha_{\ww{s}, N}(\bfz') = \sum_{\substack{f_1,\ldots,f_m\in A_+ \\ \frac{1}{r}\log_q(N+1)\ge \deg f_1 > \cdots > \deg f_m}} \sum_{\substack{N_1+\cdots+N_m=N \\ N_1,\ldots, N_m\ge 0}}\Ccal_{s_1,f_1,N_1}(\bfz')\cdots \Ccal_{s_m, f_m, N_m}(\bfz')\in \mathcal{L}_{K,\wt(\ww{s}) - N}^{(r)}.
    \]
\end{proof}

\begin{thm}\label{thm:t-expansion-of-MES}
    Fix an integer $r\ge 1$ and denote by $\bfz = (z_1, \bfz')$ an arbitrary point varying in $\Omega^{r+1}$.
    For any index $\ww{s}$, write
    \[
    E_{r+1}(\ww{s};\bfz)=\sum_{N\geq 0}\beta_{\ww{s},N}(\bfz')\cdot t_{\Lamz}(z_1)^{N}.
    \]
    Then 
    \[
    \beta_{\ww{s},N}(\bfz')\in\Lcal^{(r)}_{K,\wt(\ww{s})-N}
    \]
    for all $N\geq 0$.
\end{thm}
\begin{proof}
    We may assume that $\ww{s}=(s_1,\ldots,s_m)$ is nonempty.
    By \cref{prop.Goss expansion of MES}, we have
    \[
    E_{r+1}(\ww{s};\bfz)
    =
    \sum_{i=0}^{m}
    E_r(\ww{s}^{(i)};\bfz')
    G_{r+1}(\ww{s}_{(i)};\bfz).
    \]
    For each $0\leq i\leq m$, write
    \[
    G_{r+1}(\ww{s}_{(i)};\bfz)
    =
    \sum_{N\geq0}
    \alpha_{\ww{s}_{(i)},N}(\bfz')
    t_{\Lamz}(z_1)^N.
    \]
    By \cref{prop:t-expansion of MGS}, we have
    \[
    \alpha_{\ww{s}_{(i)},N}(\bfz')
    \in
    \Lcal_{K,\wt(\ww{s}_{(i)})-N}^{(r)}.
    \]
    Moreover,
    \[
    E_r(\ww{s}^{(i)};\bfz')
    \in
    \Zcal_{K,\wt(\ww{s}^{(i)})}^{(r)}
    \subseteq
    \Lcal_{K,\wt(\ww{s}^{(i)})}^{(r)}.
    \]
    Hence
    \[
    E_r(\ww{s}^{(i)};\bfz')
    \alpha_{\ww{s}_{(i)},N}(\bfz')
    \in
    \Lcal_{K,\wt(\ww{s})-N}^{(r)}.
    \]
    Comparing the coefficients of $t_{\Lamz}(z_1)^N$, we obtain
    \[
    \beta_{\ww{s},N}(\bfz')
    =
    \sum_{i=0}^{m}
    E_r(\ww{s}^{(i)};\bfz')
    \alpha_{\ww{s}_{(i)},N}(\bfz')
    \in
    \Lcal_{K,\wt(\ww{s})-N}^{(r)},
    \]
    as desired.
\end{proof}

\subsection{Proof of \texorpdfstring{\cref{thm.main-thm}~(i)}{Theorem 1.5 (i)}}\label{sec.proof of 1.5-(i)}

We proceed by induction on $r$.
The case $r=1$ is exactly \cref{thm.Goncharov-Zagier conjecture for function field}~(i).
Let $r\ge 1$ and assume that the assertions hold for $r$.
Suppose that we have a $\ov{K}$-linear relation
\begin{align}\label{eq:linear-comb-different-weight}
   \sum_{i=1}^m
   \sum_{\wt(\ww{s})=w_i}
   c_{\ww{s}}E_{r+1}(\ww{s};\bfz)=0
\end{align}
for some distinct $w_1,\ldots,w_m\geq 0$ and
$c_{\ww{s}}\in\ov{K}$. It suffices to show that, for each
$1\leq i\leq m$,
\[
\sum_{\wt(\ww{s})=w_i}
c_{\ww{s}}E_{r+1}(\ww{s};\bfz)=0.
\]

By \eqref{eq:linear-comb-different-weight}, we obtain
\[
\sum_{N\geq 0}\sum_{i=1}^m
\sum_{\wt(\ww{s})=w_i}
c_{\ww{s}}\beta_{\ww{s},N}(\bfz')
\cdot t_{\Lamz}(z_1)^N=0.
\]
By the uniqueness of the $t_{\Lamz}(z_1)$-expansion
(see \cite[Proposition~5.4]{BBP24}), it follows that
\begin{align}\label{eq:linear-comb-different-weight-coefficient}
 \sum_{i=1}^m
 \sum_{\wt(\ww{s})=w_i}
 c_{\ww{s}}\beta_{\ww{s},N}(\bfz')=0
\end{align}
for all $N\geq 0$.

Fix $N\geq 0$. By \cref{thm:t-expansion-of-MES},
\[
\beta_{\ww{s},N}(\bfz')
\in
\Lcal_{K,\wt(\ww{s})-N}^{(r)}.
\]
Hence, after choosing $M\gg 1$ sufficiently large, we may write
\[
\beta_{\ww{s},N}(\bfz')
=
\Delta_r^{-M}\Xi_{\ww{s},N}(\bfz')
\]
for all $\ww{s}$ occurring in
\eqref{eq:linear-comb-different-weight-coefficient}, where
\[
\Xi_{\ww{s},N}(\bfz')
\in
\Zcal_{\ov{K},\wt(\ww{s})-N+M(q^r-1)}^{(r)}.
\]
Thus, after multiplying
\eqref{eq:linear-comb-different-weight-coefficient} by $\Delta_r^M$, we obtain
\[
\sum_{i=1}^m
\sum_{\wt(\ww{s})=w_i}
c_{\ww{s}}\Xi_{\ww{s},N}(\bfz')=0.
\]
For each $i$, note that the inner sum belongs to
$\Zcal_{\ov{K},w_i-N+M(q^r-1)}^{(r)}$
and the weights
\[
w_i-N+M(q^r-1)
\]
are all distinct. Therefore, by the induction hypothesis,
\[
\sum_{\wt(\ww{s})=w_i}
c_{\ww{s}}\Xi_{\ww{s},N}(\bfz')=0
\]
for each $1\leq i\leq m$. Dividing by $\Delta_r^M$, we obtain
\[
\sum_{\wt(\ww{s})=w_i}
c_{\ww{s}}\beta_{\ww{s},N}(\bfz')=0
\]
for every $N\geq 0$ and every $1\leq i\leq m$. Hence, by the
$t_{\Lamz}(z_1)$-expansion,
\[
\sum_{\wt(\ww{s})=w_i}
c_{\ww{s}}E_{r+1}(\ww{s};\bfz)=0
\]
for each $1\leq i\leq m$, as desired.

\subsection{Proof of \texorpdfstring{\cref{thm.main-thm}~(ii)}
{Theorem 1.5 (ii)}}\label{sec.proof of 1.5(ii)}

The surjectivity of the natural map
\[
\Zcal_K^{(r)}\otimes_K\ov K
\longrightarrow
\Zcal_{\ov K}^{(r)}
\]
is immediate from the definition. We prove its injectivity by induction
on $r$. The case $r=1$ follows from
\cref{thm.Goncharov-Zagier conjecture for function field}~(ii).
Let $r\geq1$ and assume that the assertion holds in rank $r$.
It suffices to show that every $K$-linearly independent family
\[
F_1,\ldots,F_m\in\Zcal_K^{(r+1)}
\]
remains linearly independent over $\ov K$.

By \cref{thm:t-expansion-of-MES}, for each $1\leq i\leq m$, we may write
\[
F_i(\bfz)
=
\sum_{N\geq0}
\Delta_r^{-M(N)}\gamma_{i,N}(\bfz')
t_{\Lamz}(z_1)^N,
\]
where, for each $N\geq0$, the integer $M(N)\geq0$ is chosen sufficiently
large so that 
\[
\gamma_{i,N}(\bfz')\in\Zcal_K^{(r)}
\]
for all $1\leq i\leq m$. For each $N\geq0$, define the $K$-linear map
\[
\psi_N:K^{\oplus m}\longrightarrow\Zcal_K^{(r)},
\qquad
(c_1,\ldots,c_m)
\longmapsto
\sum_{i=1}^m c_i\gamma_{i,N}.
\]
Since $\Delta_r$ is nowhere vanishing, the uniqueness of the
$t_{\Lambda_{\bfz'}}$-expansion and the $K$-linear independence of
$F_1,\ldots,F_m$ imply that
\[
\bigcap_{N\geq0}\ker\psi_N=\{0\}.
\]
Since $K^{\oplus m}$ is finite-dimensional, there exists $N_0\geq0$ such that
\[
\bigcap_{N=0}^{N_0}\ker\psi_N=\{0\}.
\]
Hence the $K$-linear map
\[
\Psi:K^{\oplus m}
\longrightarrow
\bigl(\Zcal_K^{(r)}\bigr)^{\oplus(N_0+1)},
\qquad
\mathbf{c}
\longmapsto
\bigl(
\psi_0(\mathbf{c}),\ldots,\psi_{N_0}(\mathbf{c})
\bigr)
\]
is injective.

By the induction hypothesis,
\[
\Zcal_K^{(r)}\otimes_K\ov K
\cong
\Zcal_{\ov K}^{(r)}.
\]
Since $\ov K$ is flat over $K$, extending scalars from $K$ to $\ov K$
shows that
\[
\Psi_{\ov K}:
\ov K^{\oplus m}
\longrightarrow
\bigl(\Zcal_{\ov K}^{(r)}\bigr)^{\oplus(N_0+1)},
\]
defined by
\[
(c_1,\ldots,c_m)
\longmapsto
\left(
\sum_{i=1}^m c_i\gamma_{i,0},
\ldots,
\sum_{i=1}^m c_i\gamma_{i,N_0}
\right),
\]
is injective.

Now suppose that $c_1,\ldots,c_m\in\ov K$ satisfy
\[
\sum_{i=1}^m c_iF_i=0.
\]
Comparing the coefficients of
$t_{\Lambda_{\bfz'}}(z_1)^N$ and multiplying by
$\Delta_r^{M(N)}$, we obtain
\[
\sum_{i=1}^m c_i\gamma_{i,N}=0
\]
for every $N\geq0$. In particular, we have
\[
(c_1,\ldots,c_m)\in\ker\Psi_{\ov K}=0.
\]
Thus, $F_1,\ldots,F_m$ are $\ov K$-linearly independent. This proves
the injectivity of
\[
\Zcal_K^{(r+1)}\otimes_K\ov K
\longrightarrow
\Zcal_{\ov K}^{(r+1)},
\]
and completes the induction.

\begin{appendix}

\section{\texorpdfstring{\cref{thm.main-thm}}{Theorem 1.5} over \texorpdfstring{$\mathbb{C}_\infty$}{C∞}}\label{appendix}

In this appendix, we aim to prove the following stronger version of \cref{thm.main-thm} by adopting the key ideas from~\cite{BK26}, particularly the notion of ``shifted" MES in \cref{df.shifted MES}.

\begin{thm} \label{thm.main-thm over C_infty}
    Fix a positive integer $r\geq 2$.
    Then the following results hold.
    \begin{enumerate}
        \item $\Zcal^{(r)}_{\CC_\infty}$ is a graded algebra, i.e.,  \[\Zcal^{(r)}_{\CC_\infty} = \bigoplus_{w=0}^{\infty}\Zcal^{(r)}_{\CC_\infty,w}.\]
        \item The natural map $\mathcal{Z}^{(r)}_{K}\otimes_K \CC_\infty\twoheadrightarrow  \mathcal{Z}^{(r)}_{\CC_\infty}$ is an isomorphism of $\CC_\infty$-algebras. 
    \end{enumerate}
\end{thm}
Note that when $r=1$, MES are Thakur's MZVs, and the corresponding base-change statement over $\CC_\infty$ does not hold. We first mention the following immediate consequence concerning the non-vanishing of rank two MGS.
\begin{cor}\label{cor.non-vanishing of MGS (rank 2)}
    For any index $\ww{s}$, 
    \[
    G_2(\ww{s}; \mathbf{z})\ne 0.
    \]
\end{cor}
\begin{proof}
    If $G_2(\ww{s}; \mathbf{z}) = 0$, then 
    \[
    E_2(\ww{s}; \mathbf{z}) = \sum_{i=0}^{\dep(\ww{s})-1}\zeta_A(\ww{s}^{(i)})G_2(\ww{s}_{(i)}; \mathbf{z})\in \sum_{w=0}^{\wt(\ww{s})-1}\mathcal{Z}^{(2)}_{\mathbb{C}_{\infty}, w}.
    \]
    By \cref{thm.main-thm over C_infty}, this implies
\[
E_2(\ww{s};\mathbf z)=0.
\]
On the other hand, by \cite[Proposition~2.10]{CCHT26}, the constant term
of the $t_{\Lamz}(z_1)$-expansion of $E_2(\ww{s};\mathbf z)$ is
$E_1(\ww{s})=\zeta_A(\ww{s})$,
which is nonzero by \cite[Theorem~4]{Tha09}. This is a contradiction.
\end{proof}

\subsection{Proof of \texorpdfstring{\cref{thm.main-thm over C_infty}~(i)}{Theorem A.1 (i)}}

For $N\geq1$, we let
\begin{equation*}
    \gamma_N
    =
    \left(
    \begin{array}{c|cc}
    I_{r-2} & 0 & 0\\
    \hline
    0 & \theta^N & -1\\
    0 & 1 & 0
    \end{array}
    \right)
    \in \SL_r(A)
    \subseteq
    \GL_r(A),
\end{equation*}
and define a new partial order $\succ_N$ on $A^r$ as follows.

\begin{df}
    For $\bm{f},\bm{g}\in A^r$, we define
    \[
    \bm{f}\succ_N0
    \quad\Longleftrightarrow\quad
    \bm{f}\gamma_N\succ0,
    \]
    and
    \[
    \bm{f}\succ_N\bm{g}
    \quad\Longleftrightarrow\quad
    \bm{f}\gamma_N\succ\bm{g}\gamma_N.
    \]
\end{df}

For $\bfz=(z_1,\ldots,z_r)\in \CC_\infty^r$, we put
\[
|\bfz|
:=
\max_{1\leq i\leq r}|z_i|.
\]

\begin{lem}\label{lem:comparison between >_N and >}
    Let $\bm{f},\bm{g}\in A^r$ be nonzero with
    $|\bm{f}|,|\bm{g}|<q^N$. Then
    \begin{enumerate}
        \item $\bm{f}\succ_N0$ if and only if $\bm{f}\succ0$.
        \item $\bm{f}\succ_N\bm{g}$ if and only if $\bm{f}\succ\bm{g}$.
    \end{enumerate}
\end{lem}

\begin{proof}
    Write $\bm{f}=(f_1,\ldots,f_r)$. We have
    \[
    \bm{f}\gamma_N
    =
    (f_1,\ldots,f_{r-2},
    \theta^Nf_{r-1}+f_r,-f_{r-1}).
    \]
    If $\ell(\bm{f})\leq r-2$, the first nonzero coordinate of
    $\bm{f}\gamma_N$ is the same as that of $\bm{f}$, so the first
    assertion is immediate. If $\ell(\bm{f})=r-1$, then
    $\deg f_r<N$, while
    \[
    \deg(\theta^Nf_{r-1})
    =
    N+\deg f_{r-1}\geq N.
    \]
    Hence $\theta^Nf_{r-1}+f_r$ is monic if and only if
    $f_{r-1}$ is monic. Finally, if $\ell(\bm{f})=r$, then
    \[
    \bm{f}\gamma_N=(0,\ldots,0,f_r,0),
    \]
    so again $\bm{f}\succ_N0$ if and only if $\bm{f}\succ0$.
    This proves (i). The proof of (ii) follows from a similar argument.
\end{proof}

Using the partial order $\succ_N$, we define a ``shifted'' MES as follows.

\begin{df}\label{df.shifted MES}
    For any nonempty index $\ww{s}=(s_1,\ldots,s_m)$, we define
    \[
    E_r^{\langle N\rangle}(\ww{s};\mathbf{z})
    =
    \sum_{\substack{
    \bm{f}_1,\ldots,\bm{f}_m\in A^r\\
    \bm{f}_1\succ_N\cdots\succ_N\bm{f}_m\succ_N0
    }}
    \frac{1}{
    \langle\bm{f}_1,\mathbf{z}\rangle^{s_1}
    \cdots
    \langle\bm{f}_m,\mathbf{z}\rangle^{s_m}
    }.
    \]
    By convention, we put $E_r^{\ang{N}}(\varnothing;\bfz)=1$.
\end{df}

The next proposition compares
$E_r^{\langle N\rangle}(\ww{s};\mathbf{z})$
with the original MES. It shows that the shifted
series converges to $E_r(\ww{s};\mathbf{z})$ as $N\to\infty$ and relates
it to the action of $\gamma_N$.

\begin{prop}\label{prop.shifted MES converges to MES}
    Let $\ww{s}$ be an index.
    \begin{enumerate}
        \item
        For every $\mathbf{z}\in\Omega^r$, there exists a constant
        $c(\mathbf{z})>0$ such that
        \[
        \left|
        E_r^{\langle N\rangle}(\ww{s};\mathbf{z})
        -
        E_r(\ww{s};\mathbf{z})
        \right|
        \leq
        c(\mathbf{z})q^{-N}
        \]
        for every $N\geq1$.
    
        \item
        For $\mathbf{z}=(z_1,\ldots,z_{r-1},1)\in\Omega^r$, put
        \[
        j_N(\mathbf{z})
        :=
        \theta^N-z_{r-1}.
        \]
        Then
        \[
        \mathbf{z}_N
        :=
        \left(
        \frac{z_1}{j_N(\mathbf{z})},
        \ldots,
        \frac{z_{r-2}}{j_N(\mathbf{z})},
        \frac{1}{j_N(\mathbf{z})},
        1
        \right)\in\Omega^r,
        \]
        and
        \[
        E_r(\ww{s};\mathbf{z}_N)
        =
        j_N(\mathbf{z})^{\wt(\ww{s})}
        E_r^{\langle N\rangle}(\ww{s};\mathbf{z}).
        \]
    \end{enumerate}
\end{prop}

\begin{proof}
    We may assume $\ww{s}=(s_1,\ldots,s_m)$ is nonempty. By
\cref{lem:comparison between >_N and >}, the two summation sets
defining $E_r^{\langle N\rangle}(\ww{s};\mathbf{z})$ and
$E_r(\ww{s};\mathbf{z})$ agree whenever
\[
|\bm{f}_1|,\ldots,|\bm{f}_m|<q^N.
\]
Hence their difference is supported on tuples for which
$|\bm{f}_i|\geq q^N$ for some $1\leq i\leq m$.
On the other hand, notice that for each fixed $\mathbf{z}\in\Omega^r$, the $K_\infty$-linear map
\[
K_\infty^r\longrightarrow\CC_\infty,
\qquad
\bm{f}\longmapsto\langle\bm{f},\mathbf{z}\rangle,
\]
is injective. Since $K_\infty^r$ is finite-dimensional, there exists
$c_0(\mathbf{z})>0$ such that
\[
|\langle\bm{f},\mathbf{z}\rangle|
\geq
c_0(\mathbf{z})|\bm{f}|
\qquad
(\bm{f}\in K_\infty^r).
\]
Thus, a standard estimate shows that there exists a constant $c(\bfz)>0$, independent of $N$, such that
\[
\left|
E_r^{\langle N\rangle}(\ww{s};\mathbf{z})
-
E_r(\ww{s};\mathbf{z})
\right|
\leq
c(\mathbf{z})q^{-N}.
\]
This proves (i).
    
    For (ii), one checks that
    \[
    \gamma_N\mathbf{z}_N^{\tr}
    =
    j_N(\mathbf{z})^{-1}\mathbf{z}^{\tr},
    \]
    which implies $\mathbf{z}_{N}\in \Omega^{r}$.
    It follows that
    \begin{align*}
    E_r(\ww{s};\mathbf{z}_N)
    &=
    \sum_{\substack{
    \bm{f}_1,\ldots,\bm{f}_m\in A^r\\
    \bm{f}_1\succ\cdots\succ\bm{f}_m\succ0
    }}
    \frac{1}{
    (\bm{f}_1\mathbf{z}_N^{\tr})^{s_1}
    \cdots
    (\bm{f}_m\mathbf{z}_N^{\tr})^{s_m}
    }
    \\
    &=
    j_N(\mathbf{z})^{\wt(\ww{s})}
    \sum_{\substack{
    \bm{f}_1,\ldots,\bm{f}_m\in A^r\\
    \bm{f}_1\succ\cdots\succ\bm{f}_m\succ0
    }}
    \frac{1}{
    (\bm{f}_1\gamma_N^{-1}\mathbf{z}^{\tr})^{s_1}
    \cdots
    (\bm{f}_m\gamma_N^{-1}\mathbf{z}^{\tr})^{s_m}
    }
    \\
    &=
    j_N(\mathbf{z})^{\wt(\ww{s})}
    \sum_{\substack{
    \bm{f}_1,\ldots,\bm{f}_m\in A^r\\
    \bm{f}_1\succ\cdots\succ\bm{f}_m\succ0
    }}
    \frac{1}{
    \ang{\bm{f}_1\gamma_N^{-1},\mathbf{z}}^{s_1}
    \cdots
    \ang{\bm{f}_m\gamma_N^{-1},\mathbf{z}}^{s_m}
    }
    \\
    &=
    j_N(\mathbf{z})^{\wt(\ww{s})}
    \sum_{\substack{
    \bm{f}_1,\ldots,\bm{f}_m\in A^r\\
    \bm{f}_1\succ_N\cdots\succ_N\bm{f}_m\succ_N0
    }}
    \frac{1}{
    \langle\bm{f}_1,\mathbf{z}\rangle^{s_1}
    \cdots
    \langle\bm{f}_m,\mathbf{z}\rangle^{s_m}
    }
    \\
    &=
    j_N(\mathbf{z})^{\wt(\ww{s})}
    E_r^{\langle N\rangle}(\ww{s};\mathbf{z}),
    \end{align*}
    where the fourth equality follows by the change of variables
    $\bm{f}_i\mapsto\bm{f}_i\gamma_N^{-1}$ and the definition of
    $\succ_N$. This proves (ii).
\end{proof}

\begin{proof}[Proof of \cref{thm.main-thm over C_infty}~(i)]
    Suppose that
    \[
    F_0+F_1+\cdots+F_w=0,
    \]
    where $F_i\in\Zcal_{\CC_\infty,i}^{(r)}$ for every $0\leq i\leq w$.
    For each $i$, choose an expression
    \[
    F_i(\bfz)
    =
    \sum_{j=1}^{m_i}
    c_{i,j}E_r(\ww{s}_{i,j};\bfz),
    \]
    where $c_{i,j}\in\CC_\infty$ and
    $\wt(\ww{s}_{i,j})=i$, and put
    \[
    F_i^{\ang{N}}(\bfz)
    :=
    \sum_{j=1}^{m_i}
    c_{i,j}E_r^{\ang{N}}(\ww{s}_{i,j};\bfz).
    \]
    By \cref{prop.shifted MES converges to MES}, as $N\to\infty$,
    \[
    F_i^{\ang{N}}(\bfz)\longrightarrow F_i(\bfz).
    \]
    
    Evaluating the given relation at $\bfz_N$ and applying
    \cref{prop.shifted MES converges to MES} again, we obtain
    \[
    0
    =
    \sum_{i=0}^{w}F_i(\bfz_N)
    =
    \sum_{i=0}^{w}
    j_N(\bfz)^iF_i^{\ang{N}}(\bfz).
    \]
    Dividing by $j_N(\bfz)^w$ yields
    \[
    F_w^{\ang{N}}(\bfz)
    =
    -\sum_{i=0}^{w-1}
    j_N(\bfz)^{i-w}F_i^{\ang{N}}(\bfz).
    \]
    Since $F_i^{\ang{N}}(\bfz)\to F_i(\bfz)$ as $N\to\infty$, the sequences
    $\{F_i^{\ang{N}}(\bfz)\}_{N\geq1}$ are bounded. Moreover, for $N\gg0$,
    \[
    |j_N(\bfz)|^{-1}
    =
    |\theta^N-z_{r-1}|^{-1}
    =
    q^{-N}\longrightarrow0.
    \]
    Hence, letting $N\to\infty$ in the preceding identity yields
    \[
    F_w(\bfz)=0.
    \]
    Repeating the same argument for
    $F_0+\cdots+F_{w-1}=0$, we obtain
    \[
    F_0=F_1=\cdots=F_w=0.
    \]
    Therefore,
    \[
    \Zcal_{\CC_\infty}^{(r)}
    =
    \bigoplus_{w\geq0}\Zcal_{\CC_\infty,w}^{(r)}.
    \]
\end{proof}

\subsection{Proof of \texorpdfstring{\cref{thm.main-thm over C_infty}~(ii)}{Theorem A.1 (ii)}}
Recall that in \cref{rmk:rank 1}, $\wtd{\pi}$ denotes a fixed fundamental period of the Carlitz $A$-module. For an index $\ww{s}$, let
\[
\wtd{\zeta}_A(\ww{s}) = \wtd{\pi}^{-\wt{(\ww{s})}}\zeta_A(\ww{s}),\qquad \wtd{G}_2(\ww{s};\bfz) = \wtd{\pi}^{-\wt(\ww{s})} G_2(\ww{s};\bfz),\qquad \wtd{E}_2(\ww{s};\bfz) = \wtd{\pi}^{-\wt(\ww{s})} E_2(\ww{s};\bfz)
\]
be the normalized MZV, MGS and MES of rank two.
By \cref{prop.Goss expansion of MES}, we have
\begin{equation}\label{eq.normalized MES Goss expansion}
    \wtd{E}_2(\ww{s};\bfz) = \wtd{\zeta}_A(\ww{s}) + \sum_{i=1}^{\dep(\ww{s})} \wtd{\zeta}_A(\ww{s}^{(i)})\wtd{G}_2(\ww{s}_{(i)};\bfz).
\end{equation}

For any subfield $L$ of $\CC_\infty$ containing $K$ and integer $w\ge 0$, we define
\[
\wtd{\Zcal}_{L,w}^{(2)} := \Span_L\{\wtd{E}_2(\ww{s};\bfz) \mid \wt{(\ww{s})} = w\},\qquad \wtd{\Gcal}_{L,\le w}^{(2)} := \Span_L\{\wtd{G}_2(\ww{s};\bfz) \mid \wt{(\ww{s})} \le w\}.
\]
For $w\ge 0$, we further define
\[
\gr_w\wtd{\Gcal}_{L}^{(2)} = \wtd{\Gcal}_{L,\le w}^{(2)}/\wtd{\Gcal}_{L,\le w-1}^{(2)}
\]
where $\wtd{\Gcal}_{L,\le -1}^{(2)}:=0$.
We note that we have a natural $L$-vector space isomorphism given by
\begin{equation}\label{eq. isom of nonnormalized and normalized MES}
    \Zcal_{L,w}^{(2)} \xrightarrow{\sim} \wtd{\Zcal}_{L,w}^{(2)},\qquad E_2(\ww{s};\bfz) \mapsto \wtd{E}_2(\ww{s};\bfz).
\end{equation}

\begin{prop}\label{prop.Gcal is defined over K}
    For every integer $w\ge 0$, the natural map
    \[
    \wtd{\Gcal}_{K,\le w}^{(2)} \otimes_K \CC_\infty \xrightarrow{\sim} \wtd{\Gcal}_{\CC_\infty, \le w}^{(2)}
    \]
    is an isomorphism.
\end{prop}
\begin{proof}
    Notice that for any $\bfz\in \Omega^2$, $\Lamz=A$. Let $C=\phi^{\wtd{\pi}A}$ denote the Carlitz $A$-module and $\exp_C(X)=\exp_{\wtd{\pi}A}(X)\in K[\![X]\!]$ denote the Carlitz exponential (see \cite[\S3]{goss1996basic}). 
    By \cref{df.Goss polynomial,prop.G_a(t(X)) = sum 1/(X+λ)^a}, we have
    \[
    G_s^{\wtd{\pi}A}(X)\in K[X]
    \qquad
    (s\geq1),\qquad\text{and}\qquad
    G_s^{\wtd{\pi}A}
    \bigl(t_{\wtd{\pi}A}(\wtd{\pi}fz_1)\bigr)
    =
    \wtd{\pi}^{-s}
    G_s^A\bigl(t_A(fz_1)\bigr).
    \]
    On the other hand, for every $f\in A_+$, we have
    \[
    t_{\wtd{\pi}A}(\wtd{\pi}fz_1)
    =
    \frac{1}{\exp_{\wtd{\pi}A}(\wtd{\pi}fz_1)}
    =
    \frac{1}{C_f(\exp_{\wtd{\pi}A}(\wtd{\pi}z_1))}
    =
    \frac{1}{
    C_f\bigl(t_{\wtd{\pi}A}(\wtd{\pi}z_1)^{-1}\bigr)
    }\in
    K[\![t_{\wtd{\pi}A}(\wtd{\pi}z_1)]\!].
    \]
    It follows that for every nonempty index
    $\ww{s}=(s_1,\ldots,s_m)$,
    \[
    \wtd{G}_2(\ww{s};\bfz)= \sum_{\substack{f_1,\ldots,f_m\in A_+ \\ \deg f_1 > \cdots > \deg f_m}}G_{s_1}^{\wtd{\pi}A}(t_{\wtd{\pi}A}(\wtd{\pi}f_1z_1))\cdots G_{s_m}^{\wtd{\pi}A}(t_{\wtd{\pi}A}(\wtd{\pi}f_mz_1))
    \in
    K[\![t_{\wtd{\pi}A}(\wtd{\pi}z_1)]\!].
    \]
    Here, for each $M\geq 0$, \eqref{eq.ord_r(t(fz))} shows that only finitely many terms in the series above contribute to its $t_{\wtd{\pi}A}(\wtd{\pi}z_1)^M$-coefficient. Hence this coefficient belongs to $K$.
    Therefore, we get an injective $K$-linear map
    \[
    \widetilde{\mathcal{G}}_{K,\le w}^{(2)}\hookrightarrow K[\![t_{\wtd{\pi}A}(\wtd{\pi}z_1)]\!].
    \]
    Since $\CC_\infty$ is flat over $K$, 
    \[
    \widetilde{\mathcal{G}}_{K,\le w}^{(2)}\otimes_K \mathbb{C}_{\infty}\hookrightarrow \mathbb{C}_{\infty}[\![t_{\wtd{\pi}A}(\wtd{\pi}z_1)]\!]
    \]
    is still an injection.
    Notice that the image $\widetilde{\mathcal{G}}_{K,\le w}^{(2)}\otimes_K \mathbb{C}_{\infty}\to \mathbb{C}_{\infty}[\![t_{\wtd{\pi}A}(\wtd{\pi}z_1)]\!]$ is $\widetilde{\mathcal{G}}_{\mathbb{C}_{\infty},\le w}$, so the assertion follows.
\end{proof}

The following two consequences of~\cref{prop.Gcal is defined over K} are useful for the remaining proof:
\begin{enumerate}
    \item For $w\ge 1$, consider the exact sequence
    \[
    0\to \wtd{\Gcal}_{K,\leq w-1}^{(2)}\to \wtd{\Gcal}_{K,\leq w}^{(2)}\to \gr_w\wtd{\Gcal}_{K}^{(2)}\to 0.
    \]
    By~\cref{prop.Gcal is defined over K}, tensoring with $\CC_\infty$ shows that the natural map
    \begin{equation}\label{eq. gr define over K}
        \gr_w\wtd{\Gcal}_{K}^{(2)}\otimes_K\CC_\infty \xrightarrow{\sim} \gr_w\wtd{\Gcal}_{\CC_\infty}^{(2)}
    \end{equation}
    is an isomorphism.
    \item For $w\ge 1$, we have
    \begin{equation}\label{eq.intersection property}
        \wtd{\Gcal}_{K,\le w}^{(2)} \cap \wtd{\Gcal}_{\CC_\infty,\le w-1}^{(2)} = \wtd{\Gcal}_{K,\le w-1}^{(2)}.
    \end{equation}
    To see this, we first note that the inclusion ``$\supseteq$'' is clear.
    On the other hand, suppose $x \in \wtd{\Gcal}_{K,\le w}^{(2)} \cap \wtd{\Gcal}_{\CC_\infty,\le w-1}^{(2)}$.
    Since $\CC_\infty/K$ is faithfully flat, the natural map
    \[
    \gr_w\wtd{\Gcal}_K^{(2)} \to \gr_w\wtd{\Gcal}_K^{(2)}\otimes_K\CC_\infty
    \]
    is an injection.
    Composing with~\eqref{eq. gr define over K}, we obtain an injection
    \[
    \gr_w\wtd{\Gcal}_{K}^{(2)} \hookrightarrow \gr_w\wtd{\Gcal}_{K}^{(2)}\otimes_K\CC_\infty \xrightarrow{\sim} \gr_w\wtd{\Gcal}_{\CC_\infty}^{(2)}
    \]
    We see that the image of the element $x + \wtd{\Gcal}_{K,\leq w-1}^{(2)}\in \gr_w\wtd{\Gcal}_{K}^{(2)}$ under the composite map above is
    \[
    x + \wtd{\Gcal}_{\CC_\infty,\leq w-1}^{(2)} = 0\in \gr_w\wtd{\Gcal}_{\CC_\infty}^{(2)}
    \]
    and thus $x + \wtd{\Gcal}_{K,\leq w-1}^{(2)} = 0\in \gr_w\wtd{\Gcal}_{K}^{(2)}$, which gives the desired result.
\end{enumerate}

\begin{prop}\label{prop. MES iso to gr MGS}
    Let $L = K$ or $\CC_\infty$.
    Then for any integer $w\ge 0$ the map 
    \[
    \varphi_{L,w}:=\left(\wtd{E}_2(\ww{s};\bfz) \mapsto \wtd{G}_2(\ww{s};\bfz) \mod{\wtd{\Gcal}_{L,\le w-1}^{(2)}} \right) : \wtd{\Zcal}_{L,w}^{(2)} \to \gr_w\wtd{\Gcal}_{L}^{(2)}.
    \] 
    is a well-defined isomorphism.
\end{prop}
\begin{proof}
    Assume $w\geq1$ as the case $w=0$ is clear.
    We first show the well-definedness for $L = \mathbb{C}_{\infty}$.
    Suppose
    \[
    0 = \sum_{i=1}^{m} c_i\widetilde{E}_2(\ww{s}_i; \mathbf{z}) \in \wtd{\Zcal}_{\CC_\infty,w}^{(2)}
    \]
    for some $c_i\in \CC_\infty$.
    Then by~\eqref{eq.normalized MES Goss expansion} we have
    \[
    \sum_{i=1}^{m} c_i\widetilde{G}_2(\ww{s}_i; \mathbf{z}) = -\sum_{i=1}^{m} c_i \sum_{j=0}^{\dep(\ww{s}_i)-1} \widetilde{\zeta}_A((\ww{s}_i)^{(j)})\widetilde{G}_2((\ww{s}_i)_{(j)}) \in \widetilde{\mathcal{G}}^{(2)}_{\mathbb{C}_{\infty}, \le w-1}.
    \]
    Hence, $\varphi_{\mathbb{C}_{\infty}, w}$ is well-defined.
    
    Now, we show the well-definedness for $L=K$. Suppose that 
    \[
    0 = \sum_{i=1}^{m} c_i\widetilde{E}_2(\ww{s}_i; \mathbf{z}) \in \wtd{\Zcal}_{K,w}^{(2)}
    \]
    for some $c_i\in K$.
    Similarly, by~\eqref{eq.normalized MES Goss expansion} and~\eqref{eq.intersection property}, we see that
    \[
    \sum_{i=1}^{m}c_i\widetilde{G}_2(\ww{s}_i; \mathbf{z})\in \widetilde{\mathcal{G}}^{(2)}_{K, \le w}\cap \widetilde{\mathcal{G}}^{(2)}_{\mathbb{C}_{\infty}, \le w-1} = \widetilde{\mathcal{G}}^{(2)}_{K, \le w-1}
    \]
    Hence, $\varphi_{K, w}$ is well-defined.

    For $L=K$ or $L=\CC_{\infty}$, we note that the surjectivity of $\varphi_{L,w}$ is clear from the definition.
    We then show the injectivity of $\varphi_{L, w}$ as follows.
    Suppose that $\sum_{i=1}^{m}c_i\widetilde{E}_2(\ww{s}_i; \mathbf{z})\in \widetilde{\mathcal{Z}}^{(2)}_{L, w}$ such that 
    \[
    \sum_{i=1}^{m}c_i\widetilde{G}_2(\ww{s}_i; \mathbf{z}) \in \widetilde{\mathcal{G}}^{(2)}_{L, \le w-1}\subseteq\widetilde{\mathcal{G}}^{(2)}_{\CC_\infty, \le w-1}.
    \]
    By~\eqref{eq.normalized MES Goss expansion}, we have
    \[
    \sum_{i=1}^{m}c_i\widetilde{E}_2(\ww{s}_i; \mathbf{z}) -\sum_{i=1}^{m}c_i\widetilde{G}_2(\ww{s}_i; \mathbf{z}) \in \widetilde{\mathcal{G}}^{(2)}_{\CC_\infty, \le w-1},
    \]
    which implies
    \[
    \sum_{i=1}^{m}c_i\widetilde{E}_2(\ww{s}_i; \mathbf{z})\in \widetilde{\mathcal{G}}^{(2)}_{\mathbb{C}_{\infty}, \le w-1} \subseteq \sum_{j=0}^{w-1}\widetilde{\mathcal{Z}}^{(2)}_{\mathbb{C}_{\infty}, j}.
    \]
    Therefore, by~\cref{thm.main-thm over C_infty}~(i), we obtain
    \[
    \sum_{i=1}^{m}c_iE_2(\ww{s}_i; \mathbf{z})\in \sum_{j=0}^{w-1}\mathcal{Z}^{(2)}_{\mathbb{C}_{\infty}, j} \cap \Zcal_{\CC_\infty,w}^{(2)} = 0.
    \]
    This proves the injectivity.
\end{proof}

\begin{proof}[Proof of~\cref{thm.main-thm over C_infty}~(ii)]
    It suffices to show the case $r=2$.
    For general $r\geq 2$, the result follows by adapting the induction argument of~\cref{thm.main-thm}~(ii), with $\ov{K}$ replaced by $\CC_\infty$ and the base case of the induction replaced by the case $r=2$.
    
    Assume $r=2$ and $w\ge 1$.
    The map $\wtd{\Zcal}_{K,w}^{(2)}\otimes_K \CC_\infty \to \wtd{\Zcal}_{\CC_\infty,w}^{(2)}$ fits into the following commutative diagram:
    \[
    \begin{tikzcd}
    {\wtd{\Zcal}_{K,w}^{(2)}\otimes_K \CC_\infty}
    &
    {\wtd{\Zcal}_{\CC_\infty,w}^{(2)}}
    \\
    {\gr_w\wtd{\Gcal}_K^{(2)}\otimes_K \CC_\infty}
    &
    {\gr_w\wtd{\Gcal}_{\CC_\infty}^{(2)}}
    \arrow[from=1-1, to=1-2]
    \arrow["\sim", sloped, from=1-1, to=2-1]
    \arrow["\sim"', sloped, from=1-2, to=2-2]
    \arrow["\sim", from=2-1, to=2-2]
    \end{tikzcd}
    \]
    where the three isomorphisms follow from \cref{prop. MES iso to gr MGS} and~\eqref{eq. gr define over K}.
    This implies $\wtd{\Zcal}_{K,w}^{(2)}\otimes_K \CC_\infty \xrightarrow{\sim} \wtd{\Zcal}_{\CC_\infty,w}^{(2)}$ is an isomorphism.
    Therefore, by~\eqref{eq. isom of nonnormalized and normalized MES}, we obtain
    \[
    \Zcal_{K,w}^{(2)}\otimes_K \CC_\infty \xrightarrow{\sim} \Zcal_{\CC_\infty,w}^{(2)}.
    \]
    The assertion for $r=2$ follows from summing over all $w\ge 0$.
    This completes the proof.
\end{proof}
\end{appendix}

\subsection*{Acknowledgments}
The authors were partially supported by the National Science and Technology Council grant no.~115-2115-M-007-004-MY3.


\sloppy


\begin{thebibliography}{CCHT26}

\bibitem[ABP04]{ABP04}
G.~W. Anderson, W.~D. Brownawell, and M.~A. Papanikolas,
\emph{Determination of the algebraic relations among special
$\Gamma$-values in positive characteristic},
Ann. of Math. (2) \textbf{160} (2004), no.~1, 237--313.

\bibitem[AT09]{AT09}
G.~W. Anderson and D.~S. Thakur,
\emph{Multizeta values for $\mathbb{F}_q[t]$, their period interpretation,
and relations between them},
Int. Math. Res. Not. IMRN \textbf{2009} (2009), no.~11, 2038--2055.

\bibitem[And04]{And04}
Y.~Andr\'e,
\emph{Une introduction aux motifs (motifs purs, motifs mixtes, p\'eriodes)},
Panoramas et Synth\`eses, vol.~17,
Soci\'et\'e Math\'ematique de France, Paris, 2004.

\bibitem[Bac12]{bachmann2012multiple}
H.~Bachmann,
\emph{Multiple Zeta-Werte und die Verbindung zu Modulformen durch
Multiple Eisensteinreihen},
Master's thesis, Universit\"at Hamburg, 2012.

\bibitem[Bac23]{Bachmann2023StuffleRegularized}
H.~Bachmann,
\emph{Stuffle regularized multiple Eisenstein series revisited},
RIMS K\^oky\^uroku \textbf{2238} (2023), 73--86.

\bibitem[BK26]{BK26}
H.~Bachmann and H.~Kanno,
\emph{The $\mathfrak{sl}_2$-algebra structure of multiple Eisenstein series},
arXiv:2609.03777 [math.NT], 2026.

\bibitem[BW07]{BW07}
A.~Baker and G.~W\"ustholz,
\emph{Logarithmic Forms and Diophantine Geometry},
New Mathematical Monographs, vol.~9,
Cambridge University Press, Cambridge, 2007.

\bibitem[BBP24]{BBP24}
D.~Basson, F.~Breuer, and R.~Pink,
\emph{Drinfeld Modular Forms of Arbitrary Rank},
Mem. Amer. Math. Soc. \textbf{304} (2024), no.~1531.

\bibitem[BF26]{BF26}
J.~I. Burgos Gil and J.~Fres\'an,
\emph{Multiple Zeta Values: From Numbers to Motives},
Clay Math. Proc., in press, 2026.

\bibitem[Car35]{Car35}
L.~Carlitz,
\emph{On certain functions connected with polynomials in a Galois field},
Duke Math. J. \textbf{1} (1935), no.~2, 137--168.

\bibitem[Cha14]{Cha14}
C.-Y. Chang,
\emph{Linear independence of monomials of multizeta values in positive
characteristic},
Compos. Math. \textbf{150} (2014), no.~11, 1789--1808.

\bibitem[CCHT25]{CCHT25}
T.-W. Chang, S.-Y. Chen, F.-J. Huang, and H.-C. Tsui,
\emph{On $q$-shuffle relations for multiple Eisenstein series of arbitrary
rank in positive characteristic},
arXiv:2504.18879 [math.NT], 2025.

\bibitem[CCHT26]{CCHT26}
T.-W. Chang, S.-Y. Chen, F.-J. Huang, and H.-C. Tsui,
\emph{Algebra structures of multiple Eisenstein series in positive
characteristic},
arXiv:2603.10376 [math.NT], 2026.

\bibitem[Che17]{Chen2017}
H.-J. Chen,
\emph{On shuffle of double Eisenstein series in positive characteristic},
J. Th\'eor. Nombres Bordeaux \textbf{29} (2017), no.~3, 815--825.

\bibitem[Dri74]{Drinfeld1974}
V.~G. Drinfeld,
\emph{Elliptic modules},
Math. USSR-Sb. \textbf{23} (1974), no.~4, 561--592.

\bibitem[GKZ06]{gkz2006double}
H.~Gangl, M.~Kaneko, and D.~Zagier,
\emph{Double zeta values and modular forms},
in \emph{Automorphic Forms and Zeta Functions},
World Scientific, Hackensack, NJ, 2006, pp.~71--106.

\bibitem[Gek88]{Gekeler88}
E.-U. Gekeler,
\emph{On the coefficients of Drinfeld modular forms},
Invent. Math. \textbf{93} (1988), no.~3, 667--700.

\bibitem[Gek25]{Gek25-DMF-VII}
E.-U. Gekeler,
\emph{On Drinfeld modular forms of higher rank VII: Expansions at the boundary},
J. Number Theory \textbf{269} (2025), 260--340.

\bibitem[Gon01]{Goncharov01}
A.~B. Goncharov,
\emph{Multiple $\zeta$-values, Galois groups, and geometry of modular
varieties},
in \emph{European Congress of Mathematics, Vol.~I (Barcelona, 2000)},
Progr. Math., vol.~201,
Birkh\"auser, Basel, 2001, pp.~361--392.

\bibitem[Gos80]{Gos1980}
D.~Goss,
\emph{The algebraist's upper half-plane},
Bull. Amer. Math. Soc. (N.S.) \textbf{2} (1980), no.~3, 391--415.

\bibitem[Gos96]{goss1996basic}
D.~Goss,
\emph{Basic Structures of Function Field Arithmetic},
Ergebnisse der Mathematik und ihrer Grenzgebiete, 3.~Folge, vol.~35,
Springer-Verlag, Berlin, 1996.

\bibitem[Hof92]{Hoffman92}
M.~E. Hoffman,
\emph{Multiple harmonic series},
Pacific J. Math. \textbf{152} (1992), no.~2, 275--290.

\bibitem[Pel25]{Pel2025}
F.~Pellarin,
\emph{The Analytic Theory of Vectorial Drinfeld Modular Forms},
Mem. Amer. Math. Soc. \textbf{312} (2025), no.~1581.

\bibitem[Tha04]{thakur2004function}
D.~S. Thakur,
\emph{Function Field Arithmetic},
World Scientific Publishing Co., River Edge, NJ, 2004.

\bibitem[Tha09]{Tha09}
D.~S. Thakur,
\emph{Power sums with applications to multizeta and zeta zero distribution
for $\mathbb{F}_q[t]$},
Finite Fields Appl. \textbf{15} (2009), no.~4, 534--552.

\bibitem[Tha10]{thakur2010shuffle}
D.~S. Thakur,
\emph{Shuffle relations for function field multizeta values},
Int. Math. Res. Not. IMRN \textbf{2010} (2010), no.~11, 1973--1980.

\bibitem[Wus89]{Wus89}
G.~W\"ustholz,
\emph{Algebraische Punkte auf analytischen Untergruppen algebraischer Gruppen},
Ann. of Math. (2) \textbf{129} (1989), no.~3, 501--517.

\bibitem[Yu97]{Yu97}
J.~Yu,
\emph{Analytic homomorphisms into Drinfeld modules},
Ann. of Math. (2) \textbf{145} (1997), no.~2, 215--233.

\bibitem[Zag94]{Zagier94}
D.~Zagier,
\emph{Values of zeta functions and their applications},
in \emph{First European Congress of Mathematics, Vol.~II (Paris, 1992)},
Progr. Math., vol.~120,
Birkh\"auser, Basel, 1994, pp.~497--512.

\bibitem[Zha16]{Zhao2016}
J.~Zhao,
\emph{Multiple Zeta Functions, Multiple Polylogarithms and Their Special Values},
Series on Number Theory and Its Applications, vol.~12,
World Scientific Publishing Co. Pte. Ltd., Hackensack, NJ, 2016.

\end{thebibliography}
\end{document}